\documentclass[a4paper,12pt]{article}
\usepackage[utf8]{inputenc}
\usepackage[english]{babel}
\usepackage{enumitem}
\usepackage{color}
\usepackage{amsmath}
\usepackage{amssymb}
\usepackage{amsthm}
\usepackage{mathrsfs}
\usepackage{amsfonts}
\usepackage{epsfig}
\usepackage{graphicx}
\usepackage{multirow}
\usepackage{cite}
\usepackage{algorithmic}
\usepackage{algorithm}
\usepackage{bm}
\usepackage{verbatim}
\usepackage{subfigure,bm,amssymb,amsmath,color,graphicx,epsfig,psfrag,indentfirst,latexsym,booktabs,lineno}
\usepackage{indentfirst}
\usepackage{caption}
\newtheorem{theorem}{Theorem}[section]

\newtheorem{lemma}{Lemma}[section]

\newtheorem{problem}{Problem}[section]
\newtheorem{example}{Example}[section]
\newtheorem*{proof1}{Proof of Theorem 3.1}
\theoremstyle{definition}

\numberwithin{equation}{section}

\allowdisplaybreaks

\begin{document}
\title{Inverse problems for elastic wave from Partial Cauchy Data: Uniqueness and Co-inversion for Shape and Impedance Function}
\author{{Yao Sun$^{\dag}$, Yan Chang$^{\ddag,*}$, Yukun Guo$^\ddag$}\\
\small{$^*$Corresponding author: changyan@hit.edu.cn}\\
\small{$^\dag$ College of Science, Civil Aviation University of China, Tianjin, P. R. China.}\\
\small{$^\ddag$} School of Mathematics, Harbin Institute of Technology, Harbin, P. R. China}
\date{}
\maketitle{} {\bf Abstract.}
We consider an inverse problem for the elastic wave of simultaneously reconstructing the impedance and the geometric information of the bounded body that is occupied by a homogeneous and isotropic elastic medium from the measured Cauchy data. A two-stage reconstruction method is proposed to realize simultaneous reconstruction of multiple targets.  In the first step, we restore the aperture information by utilizing the observed Cauchy data that is measured on an accessible part of the boundary. In the second step, we start with the boundary condition and propose a novel iterative method to simultaneously reconstruct the missing boundary and the impedance function.
	Theoretically, we establish the uniqueness result of the co-inversion problem based on analyzing the properties of the corresponding operators. An explicit derivative is computed for the iterative method. Numerical examples are presented to test the effectiveness and efficiency of the proposed method.


 \section{Introduction}
In the vast realm of earth sciences, we are eager to lift the veil of mystery from the earth's depths, understanding the distribution of subsurface rock layers, their structural characteristics, and potential geological hazards. In the field of engineering structural health monitoring, we urgently need to accurately ascertain whether there are damages or defects within structures, as well as their specific locations and shapes, to ensure the safe and stable operation of various infrastructure facilities. The inverse problem for the elastic wave serves as a crucial key to unlocking these unknown domains \cite{1}.

Elastic waves, a manifestation of mechanical vibration propagation through matter, exhibit complex interactions with the medium's properties and the geometric configurations inherent in the Earth's subsurface or within engineering structures as they propagate. \cite{4}. When elastic waves encounter interfaces or internal structures of different shapes, their propagation paths, wave characteristics, and energy distributions undergo significant changes.
When endeavoring to deduce the pertinent physical properties of structures through the observation of wavefield information, a captivating yet formidable class of problems emerges: inverse problems.

A representative inverse problem in this realm involves ascertaining the boundary shape of an elastic body, relying on the Cauchy data gathered from the accessible part of its boundary. These Cauchy data typically encompass the displacement and traction values measured at the accessible boundary. In practical scenarios, directly measuring the shape of the entire elastic body is generally challenging or even unfeasible, particularly when the body is buried underground or constitutes a component of a large-scale engineering structure. Nevertheless, we can generally install sensors on the accessible portion of the boundary to capture the elastic wave responses. The ultimate objective of the inverse problem is to leverage this restricted yet invaluable information to reconstruct the unknown boundary shape of the elastic body.
In the context of electrostatics, the relevant model gives rise to an inverse boundary value problem associated with the Laplace equation. Similarly, for elastic obstacles, an inverse boundary value problem emerges in connection with the Navier equation. For both scenarios, interested readers can refer to relevant literature \cite{5,6,7,8}.

Despite its importance, solving this inverse problem is far from straightforward. A reason accounting for this is that the inverse problem is an ill-posed problem, which means that small errors in the measured Cauchy data can lead to large uncertainties in the reconstructed boundary shape. Therefore, developing robust and efficient numerical methods to solve this inverse problem has been an active area of research in recent decades. Among numerous inversion algorithms, the direct sampling method based on the indicator function has received continuous attention from scholars both at home and abroad. In \cite{RTM}, the author propose a reverse time migration method to reconstruct the shape of the obstacle from the scattered field. In \cite{12,13}, factorization method has been extensively researched by Hu et al. We also refer to \cite{lxd, Liu2019} for the related study on the sampling method for inverse elastic scattering problem.
Being a significant qualitative numerical method, the sampling methods are usually straightforward since only several inner products are involved in these methods and the visualization of the indicator function can be utilized to realize reconstruction.  Nevertheless, such sampling methods are usually less accurate and it may be difficult to recognize the boundary curve accurately from the reconstruction.

In contrast, the quantitative method based on the iterative method or the optimization method realize a more accurate reconstruction. By reformulating the inverse problem as an optimization problem, the quantitative method prefers to reconstruct the boundary curve or the impedance function by solving the nonlinear and ill-posed integral equations for the unknown targets. In \cite{18}, the authors propose a nonlinear integral equation method to simultaneously reconstruct the shape and impedance in corrosion detection.
In \cite{KressIP01,KressJIE08}, Kress and Rundell consider the inverse problems of time-harmonic acoustic or electromagnetic waves. A regularized Newton iteration is given to solve the inverse problem of simultaneously determining the impedance and shape of a two-dimensional scatterer, which is based on a knowledge of the far-field pattern. The main idea is to establish the ill-posed nonlinear equation for the operator that maps the boundary and the boundary impedance of the scatterer onto the far-field pattern.
Later, Kress and Rundell \cite{Kress2005IP} determine the shape of a perfectly conducting inclusion via a pair of nonlinear and ill-posed integral equations for the unknown boundary. The integral equations can be solved by the linearization corresponding to the regularized Newton iterations. This idea is to extend by Cakoni and Kress \cite{17} to a simply connected planar domain from a pair of Cauchy data on the known boundary curve. This method is also used to complete the Cauchy data on the unknown boundary curve, and then an impedance profile is given by a point to point method. Thus the boundary curve is to recover by a known impedance profile. We also refer to \cite{era, Chang23} for the optimization method or the Newton's type iterative method to simultaneous reconstructing the obstacle and its excitation sources from the total field data.

Motivated by the idea in \cite{KressIP01} proposed by Kress and Rundell, we aim to demonstrate a Newton-type iterative method for reconstructing the elastic impedance function and the shape of the elastic body based on the integral equations. Different from the previous work \cite{KressIP01} by a direct integral approach for the mixed problem, we consider an indirect integral equation by the approach of having an integral operator defined on a virtual boundary enclosing the scatterer.

Though the singularities of the kernels corresponding to the direct integral approach would provide some more information about boundary functions. We apply the "safe'' approach of having an integral operator defined on a boundary enclosing the scatterer. It should not deal with the singularities of the kernels on the boundary especially the strongly singular integral resulting the hypersingular integral from Fr$\rm\acute{e}$chet derivative. The indirect integral approach will easily give the Fr\'{e}chet derivative of the boundary integral operator.
In addition, the salient features of the current work can be summarized as follows: First, the proposed method has the ability to simultaneously and quantitatively reconstructing both the shape and the impedance function from the Cauchy data; Second, by revealing the analytical dependence of the solution on frequency, we prove the uniqueness result for the co-inversion under consideration; Third, by reconstructing the unknown displacement from the Cauchy problem first we propose a novel Newton method to realize the simultaneous reconstruction, where there is no forward solver involved and thus the method proposed is effective and efficient. Finally, the proposed method is insensitive to the measurement noise and exhibits stable reconstruction capability.

The paper is organized as follows. In section $2$, we introduce the inverse impedance problem and the inverse shape problem. In section $3$, some theoretical results including a unique result are given by introducing the Helmholtz decomposition of the displacement. In section $4$, we specifically demonstrate linearization of the integral equations and the Newton-type iterative method for solving inverse shape problem. At last, some numerical examples are provided to demonstrate the effectiveness of the proposed linearized iterative scheme.

\section{Problem Setup}
We begin this section with a brief description of the problem under consideration. Let $D\subset\mathbb{R}^2$ be an open bounded domain occupied by a homogeneous and isotropic elastic medium with density $\rho$. Let $\Gamma_0$ indicate the known portion of the boundary $\partial D$. Denote the unknown part of the boundary $\partial D$ by $\Gamma_m=\partial D\backslash \Gamma_0$ such that $\partial D=\Gamma_0\cup\Gamma_m,$ $\Gamma_0\cap\Gamma_m=\emptyset$. We refer to Figure \ref{fig: boundary} for an intuitive illustration of the geometry setup.
\begin{figure}
	\centering
	\includegraphics[width=0.5\linewidth]{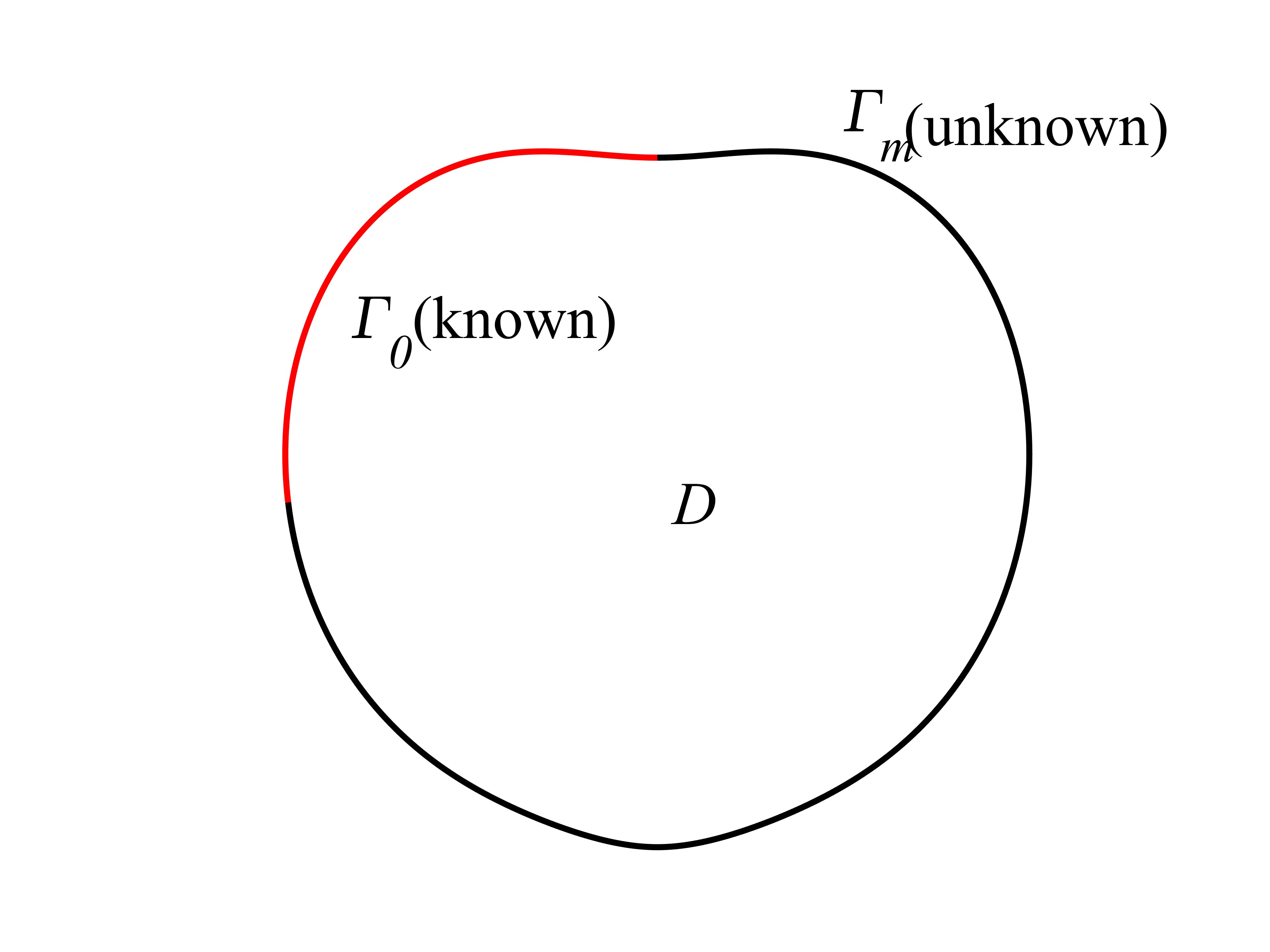}
	\caption{Geometry Setup.}
	\label{fig: boundary}
\end{figure}

Let $\omega>0$ be the angular frequency and $\lambda,\mu$ be the Lam\'{e} constants satisfying $\mu>0,\lambda+\mu>0$. Consider the following boundary value problem:
\begin{align}
	&\mu\Delta{\bm u}+(\lambda+\mu)\nabla\nabla\cdot{\bm u}+\rho\omega^2{\bm u}=0,\quad \text{in} \ D,\\
	&{\bm u}={\bm f},\quad \text{on}\ \Gamma_0,\\
	&T_{\bm n}{\bm u}+\mathrm{i}\omega\chi{\bm u}={\bm g},\quad \text{on}\ \Gamma_m,\label{eq: impedance}
\end{align}
i.e, the displacement $\bm u$ satisfies the Dirichlet boundary condition on $\Gamma_0$ and the impedance boundary condition on $\Gamma_m$ with $\chi\ge0$ being the continuous impedance function. In \eqref{eq: impedance}, $T_{\bm n}$ is the surface traction operator defined by
\begin{align}\label{eq: Tn}
	T_{\bm n}{\bm u}:=2\mu\frac{\partial\bm u}{\partial\bm n}+\lambda{\bm n}\text{div}{\bm u}-\mu{\bm n}^\bot\text{div}^\bot{\bm u},
\end{align}
where $\bm n$ is the unit outward normal vector.

The inverse problem under consideration in this paper can be formulated as:
\begin{problem}\label{ip}
	Simultaneously determine the shape of the unknown portion $\Gamma_m$ together with the impendance function $\chi$ from the Cauchy data ${\bm h}=\left({\bm f},{\bm t}\right)^\top$ that measured on $\Gamma_0$, where ${\bm f}(x)$ is the measured displacement and ${\bm t}=T_{\bm n}{\bm f}$ is the corresponding traction.
\end{problem}
\section{Uniqueness}
In this section, we focus on the uniqueness results for problem \ref{ip}.
The main result for this section is given by the following theorem:
\begin{theorem}\label{thm: unique}
	Assume that $D_1$ and $D_2$ are elastic bodies with impedance $\chi_1$ and $\chi_2,$ respectively and $\partial D_1\cap \partial D_2=\Gamma_0.$ Denote the displacements corresponding to $(D_j,\chi_j),j=1,2$ by ${\bm u}_j(x,\omega),j=1,2.$ For $\omega\in(a,b),$ if
	
	$$\bm u_1(\bm x, \omega)|_{\Gamma_0}=\bm u_2(\bm x, \omega)|_{\Gamma_0},~~ T_{\bm n}\bm u_1(\bm x, \omega)|_{\Gamma_0}= T_{\bm n}\bm u_2(\bm x, \omega)|_{\Gamma_0}$$
	Then $D_1=D_2,$ and $\chi_1=\chi_2.$
\end{theorem}

Before presenting the proof of Theorem \ref{thm: unique}, we revisit several pivotal findings concerning the analytical dependency of the solution on the angular frequency $\omega$.

Note that the Helmholtz decomposition for the displacement $\bm u$ reads
\begin{align}\label{eq: HD}
	{\bm u}=\nabla u_p+\nabla^\bot u_s.
\end{align}
By introducing a virtual boundary $\partial B$ such that $D\subset\subset B,$ the potential functions $u_p$ and $u_s$ can be represented by the single layer potentials as follows:
\begin{subequations}
	\begin{align}\label{eq: single_up}
		u_p(x)=\int_{\partial B}\Phi(k_p|x-y|)\varphi_1(y)\mathrm{d}s(y),\quad x\in D,\\\label{eq: single_us}
		u_s(x)=\int_{\partial B}\Phi(k_s|x-y|)\varphi_2(y)\mathrm{d}s(y),\quad x\in D,
	\end{align}
\end{subequations}
where $\varphi_1,\varphi_2$ are the density functions. This combining with \eqref{eq: HD} leads to
\begin{align*}
	{\bm u}(x)=\int_{\partial B}\mathbb{E}(x,y){\bm \varphi}(y)\mathrm{d}s(y),\quad x\in D,
\end{align*}
where ${\bm \varphi}=(\varphi_1,\varphi_2)^\top,$ and
$$
\mathbb{E}(x,y)=\begin{bmatrix}
	\partial_{x_1}\Phi(k_p|x-y|)&\partial_{x_2}\Phi(k_s|x-y|)\\
	\partial_{x_2}\Phi(k_p|x-y|)&-\partial_{x_1}\Phi(k_s|x-y|)
\end{bmatrix}.
$$
Under the representations \eqref{eq: HD}, \eqref{eq: single_up}, and \eqref{eq: single_us}, the corresponding traction $T_{\bm n}{\bm u}$ defined by \eqref{eq: Tn} can be further rewriten by
\begin{align*}
	\left(T_{\bm n}{\bm u}\right)(x)=\int_{\partial B}\mathbb{T}(x,y){\bm \varphi}(y)\mathrm{d}s(y),\quad x\in\partial D,
\end{align*}
where
\begin{align*}
	\mathbb{T}=\begin{bmatrix}
		\mathbb{T}_{11}&\mathbb{T}_{12}\\\mathbb{T}_{21}&\mathbb{T}_{22}
	\end{bmatrix}
\end{align*}
and
\begin{align*}
	\mathbb{T}_{1\ell}(x,y)&=\left[(\lambda+2\mu)\frac{\partial\mathbb{E}_{1\ell}(x,y)}{\partial x_1}+\lambda\frac{\partial\mathbb{E}_{2\ell}(x,y)}{\partial x_2}\right]n_1+
	\mu\left(\frac{\partial\mathbb{E}_{1\ell}(x,y)}{\partial{x_2}}+\frac{\partial\mathbb{E}_{2\ell}(x,y)}{\partial{x_1}}\right)n_2,\\
	\mathbb{T}_{2\ell}(x,y)&=\mu\left(\frac{\mathbb{E}_{1\ell}(x,y)}{\partial{x_2}}+\frac{\mathbb{E}_{2\ell}(x,y)}{\partial{x_1}}\right)n_1+\left[\lambda\frac{\partial\mathbb{E}_{1\ell}(x,y)}{\partial{x_1}}+(\lambda+2\mu)\frac{\mathbb{E}_{2\ell}(x,y)}{\partial{x_2}}\right]n_2
\end{align*}
with $\ell=1,2.$

Denote by
\begin{align*}
	&\mathcal{S}{\bm\varphi}(x):=\int_{\partial B}\mathbb{E}(x,y){\bm\varphi}(y)\mathrm{d}s(y),\quad x\in\mathbb{R}^2\backslash\partial B,\\
	&\mathbb{K}{\bm\varphi}(x):=\int_{\partial B}\mathbb{T}(x,y){\bm\varphi}(y)\mathrm{d}s(y),\quad x\in\partial D,\\
	&\mathbb{S}{\bm\varphi}(x):=\int_{\partial B}\mathbb{E}(x,y){\bm\varphi}(y)\mathrm{d}s(y),\quad x\in\partial D.
\end{align*}
Drawing on an analytical approach akin to that employed in the proof presented in \cite{Sun24}, we deduce the subsequent properties:
\begin{lemma}
	The operator $\mathbb{S}$ is compact, injective, and has dense range if $\rho\omega^2$ is not a Dirichlet eigenvalue of the negative Lam\'{e} operator in the interior of $D.$
\end{lemma}

Define the operator $\mathcal{N}$ on $\Gamma_0$ by
\begin{align}\label{eq: N}
	\mathcal{N}{\bm\varphi}(x):=\begin{bmatrix}
		\mathcal{S}{\bm\varphi}(x)\\
		\mathbb{K}{\bm\varphi}(x)
	\end{bmatrix},\quad x\in\Gamma_0,
\end{align}
then we establish some properties of $\mathcal{N}$ as follows:

\begin{theorem}
	The operator
	$\mathcal{N}: {\textbf{L}}^{2}(\partial B)\rightarrow {
		\textbf{H}}^{1/2}(\Gamma_0)\times {\textbf{H}}^{-1/2}(\Gamma_0)$ is
	compact, injective, and has a dense range.
\end{theorem}
\begin{proof}
	Since $\mathcal{N}: {\rm \textbf{\emph{L}}}^{2}(\partial B)\rightarrow {
		\textbf{\emph{H}}}^{1/2}(\Gamma_0)\times {\textbf{\emph{H}}}^{-1/2}(\Gamma_0)$ is an integral operator with continuous kernel, it is compact.
	
	We now prove the injectivity of $\mathcal{N}$. To show this, we commence by setting  $\mathcal{N}\bm \varphi=0$. Drawing upon the uniqueness property of the Cauchy problem, we can deduce that $\mathcal{S}\bm \varphi=0$ within the domain $D$. Further, the principle of unique continuation leads to the fact that $\mathcal{S}\bm \varphi=0$ in $B$.
	
	In fact
	\begin{align*}
		\mathcal{S}{\bm \varphi}({\bm{x}})=\nabla_{\bm x}\int_{\partial B}\Phi(\kappa_{p}| \bm{x}- \bm{y}|)\varphi_1(\bm y){\rm ds}(\bm y)+{\nabla_{\bm x}^\bot}\int_{\partial B}\Phi(\kappa_{s}| \bm{x}- \bm{y}|)\varphi_2(\bm y){\rm ds}(\bm y).
	\end{align*}
	Denote by
	$$\phi(\bm x)=\int_{\partial B}\Phi(\kappa_{p}| \bm{x}- \bm{y}|)\varphi_1(\bm y){\rm ds}(\bm y),~~ \bm{x}\in \mathbb{R}^{2}\setminus{\partial B},$$
	and
	$$\psi(\bm x)=\int_{\partial B}\Phi(\kappa_{s}| \bm{x}- \bm{y}|)\varphi_2(\bm y){\rm ds}(\bm y),~~ \bm{x}\in \mathbb{R}^{2}\setminus{\partial B},$$
	then
	\begin{eqnarray}\label{eq: nabla}
		\nabla\phi=-\nabla^\bot\psi,~~{\rm in}~~B.
	\end{eqnarray}
	Taking the inner product of \eqref{eq: nabla} with $\nabla$ and $\nabla^{\bot}$, respectively, gives
	\begin{align*}
		\triangle\phi=0,~~~\triangle\psi=0,~~{\rm in}~~B.
	\end{align*}
	Noting that $\phi$ and $\psi$ satisfy the Helmholtz equations, thus
	\begin{align*}
		\phi=\psi=0,~~{\rm in}~~B,
	\end{align*}
	this further leads to
	\begin{eqnarray}\label{eq: 3_9}
		\frac{\partial \phi_{-}}{\partial \nu}=\frac{\partial \psi_{-}}{\partial \nu}=0,~~{\rm on}~~\partial{B},
	\end{eqnarray}
	where $\nu$ denotes the unit outward normal vector defined on $\partial B$ and `$-$' denotes
	the limits as $\bm x\rightarrow \partial B$ from inside of $B$.
	
	From the jump relation of the single-layer potential functions (see e.g.\cite{Colton2013}), we can get
	\begin{eqnarray}\label{jump}
		\phi_{+}=\phi_{-}~~{\rm and}~~\frac{\partial \phi_{-}}{\partial \nu}-\frac{\partial \phi_{+}}{\partial \nu}=\varphi_1,~~{\rm on}~~{\partial{B}},
	\end{eqnarray}
	where $+$ denotes the limits as $\bm x\rightarrow \partial B$ from outside of $\partial B$.
	
	Combining \eqref{eq: 3_9} and (\ref{jump}) gives $\phi_{+}=0$ and $-\frac{\partial \phi_{+}}{\partial \nu}=\varphi_1$ on $\partial{B}$. This means that $\phi$ satisfies the homogeneous exterior Dirichlet boundary value
	problem with the Sommerfeld radiation condition. The uniqueness for the exterior Dirichlet problem yields $\phi = 0$ outside of $B$, and thus $\frac{\partial \phi_{+}}{\partial \nu}=0$. Combining with equation (\ref{jump}), we have $\varphi_1=0$. The same augment will give $\varphi_2=0$. Therefore
	the operator $\mathcal{N}$ is injective.
	
	Next we prove that $\mathcal{N}$ has dense range. Following the idea in \cite{17},
	we consider the adjoint
	operator $\mathcal{N}^*: {
		\emph{\textbf{H}}}^{1/2}(\Gamma_0)\times \emph{\textbf{H}}^{-1/2}(\Gamma_0)\rightarrow \emph{\textbf{L}}^{2}(\partial B)$ defined by
	\begin{align*}
		\langle\mathcal{N}\bm\varphi, [\bm \xi_1, \bm \xi_2]\rangle=\langle\bm\varphi, \mathcal{N}^*[\bm \xi_1, \bm \xi_2]\rangle.
	\end{align*}
	If we can prove that $\mathcal{N}^*$ is injective, then the conclusion that $\mathcal{N}$ has dense range follows.
	
	From the definition \eqref{eq: N}, we know that
	\begin{eqnarray*}
		\mathcal{N}^*(\bm \xi_1, \bm \xi_2)=\int_{\Gamma_0}\mathbb{E}^\top( \bm{x}, \bm{y})\overline{\bm \xi_1}(\bm x){\rm ds}(\bm x)+\int_{\Gamma_0}\mathbb{T}^\top( \bm{x}, \bm{y})\overline{\bm \xi_2}(\bm x){\rm ds}(\bm x), ~~ \bm{y}\in \partial B.
	\end{eqnarray*}
	Let $\mathcal{N}^*(\bm \xi_1, \bm \xi_2)=0$, and denote by
	\begin{align*}
		\bm w'(\bm y)=\int_{\Gamma_0}\mathbb{E}^\top( \bm{x}, \bm{y})\overline{\bm \xi_1}(\bm x){\rm ds}(\bm x), ~~ \bm{y}\in \mathbb{R}^{2}\setminus {\Gamma_0},\\
		\bm w''(\bm y)=\int_{\Gamma_0}\mathbb{T}^\top( \bm{x}, \bm{y})\overline{\bm \xi_2}(\bm x){\rm ds}(\bm x), ~~ \bm{y}\in \mathbb{R}^{2}\setminus {\Gamma_0}.
	\end{align*}	
	Define
	${\bm W}={\bm w'}+{\bm w''}=:\begin{bmatrix}W_1(\bm y)\\W_2(\bm y)
	\end{bmatrix}.$
	Then the component $W_1(\bm x)$ satisfies
	\begin{eqnarray*}
		\begin{cases}
			\triangle{W_1}+\kappa_{p}^2{W_1}=0,  ~~ {\rm in}~~\mathbb{R}^{2}\setminus{\overline{B}},\\
			W_1( \bm{y})=0,~~~{\rm on}\ \partial B.\\
		\end{cases}
	\end{eqnarray*}
	and the Sommerfeld radiation condition. The uniqueness for the exterior Dirichlet problem yields $W_1(\bm y)= 0$ outside of $B$. By the unique continuation argument, we derive $W_1(\bm y)= 0$ in $\mathbb{R}^{2}\setminus{\Gamma_0}$. Similarly, we can also prove that $W_2(\bm y)= 0$ in $\mathbb{R}^{2}\setminus{\Gamma_0}$.
	
	Denote
	\begin{eqnarray*}
		\widetilde{\bm W}(\bm y)=\nabla W_1(\bm y)+{\bf curl} W_2(\bm y),~~\bm{y}\in \mathbb{R}^{2}\setminus{\Gamma_0},
	\end{eqnarray*}
	then it holds that
	\begin{eqnarray}\label{Th2.e7}
		\widetilde{\bm W}(\bm y)_{+}=\widetilde{\bm W}(\bm y)_{-}=0~~{\rm and}~~T_{\bm n}\widetilde{\bm W}(\bm y)_{+}=T_{\bm n}\widetilde{\bm W}(\bm y)_{-}=0,~~{\rm on}~~{\Gamma_0}.
	\end{eqnarray}
	On the other hand, we have
	\begin{align}\label{Th2.e8}
		\widetilde{\bm W}(\bm y)=\rho\omega^2\int_{\Gamma_0}\Gamma(\bm x, \bm y){\overline{\bm \xi_1}}(\bm x){\rm ds}(\bm x)+\rho\omega^2\int_{\Gamma_0}T_{{\bm n}_x}\Gamma(\bm x, \bm y){\overline{\bm \xi_2}}(\bm x){\rm ds}(\bm x),
	\end{align}
	where $\Gamma(\bm x,\bm y)$ is the fundamental solution to the Navier equation (see e.g.\cite{McLean,12}).
	
	Let $\bm \alpha_1$ and $\bm \alpha_2$ be the extensions of $\overline{\bm \xi_1}$ and  $\overline{\bm \xi_2}$
	by zero to the whole boundary $\partial D$. Then we can rewrite \eqref{Th2.e8} as
	\begin{eqnarray}\label{Th2.e9}
		\widetilde{\bm W}(\bm y)=\rho\omega^2\int_{\partial D}\Gamma(\bm x, \bm y){{\bm \alpha_1}}(\bm x){\rm ds}(\bm x)+\rho\omega^2\int_{\partial D}T_{{\bm n}_x}\Gamma(\bm x, \bm y){{\bm \alpha_2}}(\bm x){\rm ds}(\bm x).\nonumber
	\end{eqnarray}
	From the jump relations of single-layer and double-layer potential functions (see e.g. \cite{Kupradze79, McLean}), we derive
	\begin{eqnarray}\label{Th2.e10}
		\widetilde{\bm W}(\bm y)_{+}&-\widetilde{\bm W}(\bm y)_{-}=\rho\omega^2{\bm \alpha_2},~~{\rm on}~~{\partial{D}},\\\label{Th2.e11}
		T_n \widetilde{\bm W}(\bm y)_{-}&-T_n \widetilde{\bm W}(\bm y)_{+}=\rho\omega^2{\bm \alpha_1},~~{\rm on}~~{\partial{D}}.
	\end{eqnarray}
	Combining (\ref{Th2.e7}), (\ref{Th2.e10}) and (\ref{Th2.e11}) gives
	$$\bm \alpha_1=\bm \alpha_2=0.$$ Therefore,
	$${\bm \xi_1}={\bm \xi_2}=0,$$
	which illustrates that  $\mathcal{N}^*$ is injective, and thus $\mathcal{N}$ has dense range.
\end{proof}

Before proving the uniqueness theorem, we should first give a property of the solution $\bm u(\bm x, \omega).$

\begin{lemma}
	The displacement $\bm u(\bm x, \omega)$ depends analytically on $\omega\in\mathbb{C}$, if $\bm h(\bm x, \omega)$ depends analytically on $\omega\in\mathbb{C}$.
\end{lemma}
\begin{proof}
	We seek the displacement field in the form as
	\begin{equation}\label{lemma1}
		\bm u(\bm x)=\int_{\partial B}\mathbb{E}(\bm{x},\bm{y})\bm\varphi(\bm y){\rm ds}(\bm y),
		~~ \bm{x}\in {D},
	\end{equation}
	with $\bm\varphi=(\varphi_1,\varphi_2)^\top$ and
	\begin{equation}
		\mathbb{E}(\bm{x},\bm{y})=
		\left(\begin{array}{cc}
			\frac{\partial \Phi(\kappa_{p}| \bm{x}- \bm{y}|)}{\partial x_1} &  \frac{\partial \Phi(\kappa_{s}| \bm{x}- \bm{y}|)}{\partial x_2} \\
			\frac{\partial \Phi(\kappa_{p}| \bm{x}- \bm{y}|)}{\partial x_2} & -\frac{\partial \Phi(\kappa_{s}| \bm{x}- \bm{y}|)}{\partial x_1} \\
		\end{array} \right).\nonumber
	\end{equation}
	Since the fundamental solution to the Helmholtz equation depends analytically on wavenumber $\kappa_p$ or $\kappa_s$ \cite{Colton2013}, together with $\kappa_{p}=\omega\sqrt{\frac{\rho}{\lambda+2\mu}}$ and $\kappa_{s}=\omega\sqrt{\frac{\rho}{\mu}}$, and thus the kernel $\mathbb{E}(\bm{x},\bm{y})$ of \eqref{lemma1} depends analytically on $\omega$.
	Given the measurement ${\bm h}=({\bm f}|_{\Gamma_0},{\bm t}|_{\Gamma_0})^\top,$ then the density function ${\bm\varphi}\in{\bm L}^2(\partial B)$ is supposed to satisfy
	\begin{align}\label{jingque}
		\mathcal{N}{\bm\varphi}={\bm h}.
	\end{align}
	 Since $\bm h(\bm x, \omega)$ depends analytically on $\omega\in\mathbb{C}$ and $\mathcal{N}$ has dense range, we know that $\bm\varphi$ depends analytically on $\omega\in\mathbb{C}$.
	From the above representation (\ref{lemma1}), it can be seen that the displacement
	field $\bm u(\bm x, \omega)$ depends analytically on $\omega$.
\end{proof}
We are now in the position to prove the main uniqueness result, which has be presented by Theorem \ref{thm: unique}.
\begin{proof1}
	Assume that $D_1$ and $D_2$ are two different  elastic bodies with impedance $\chi_1$ and $\chi_2,$ respectively that generate the same Cauchy data pair $(\bm f, \bm t)$ on $\Gamma_0$, $\Gamma_m^{(j)}=\partial D_j\backslash \Gamma_0$. Let $\bm u_j(\bm x, \omega), j=1,2$ be the solutions to
	\begin{equation}\label{Th1-1}
		\begin{split}
			\begin{cases}
				\nabla\cdot \bm \sigma(\bm u_j) +\rho\omega^2 \bm u_j=0,~~~& {\rm in}\ {D_j}, \\
				\bm u_j={\bm f(\bm x, \omega)},~~~&{\rm on}\ \Gamma_0,\\
				{ T_{n_j}\bm u_j+{\rm i}\omega{\chi_j}\bm u_j=0},~~&{\rm on}\ \Gamma_m^{(j)}.
			\end{cases}
		\end{split}
	\end{equation}
	
	First, we prove $\Gamma_m^{(1)}=\Gamma_m^{(2)}$. It should be noted that ${\bm f(\bm x, \omega)}\neq0$.  Assume on the contradiction that $\Gamma_m^{(1)}\neq\Gamma_m^{(2)}$. Following the idea in \cite{Isakov} for acoustic scattering, without loss of generality we suppose that there is a
	point in $D_1\setminus \overline{D_2}$. Otherwise we may switch the notations $D_1$ and $D_2$ if necessary. Then we can always find a domain $D^*\subset D_1$ such that
	$\partial D^*=\Lambda_1\cup\Lambda_2,\Lambda_1\subset\Gamma_m^{(1)},\Lambda_2\subset\Gamma_m^{(2)}$, $\Lambda_2\in (\overline{D_1}\cap\overline{D_2})$ (see figure \ref{fig1}).
	
	\begin{figure}[htp]
		\begin{center}
			{\includegraphics[angle=0,
				width=0.4\linewidth]{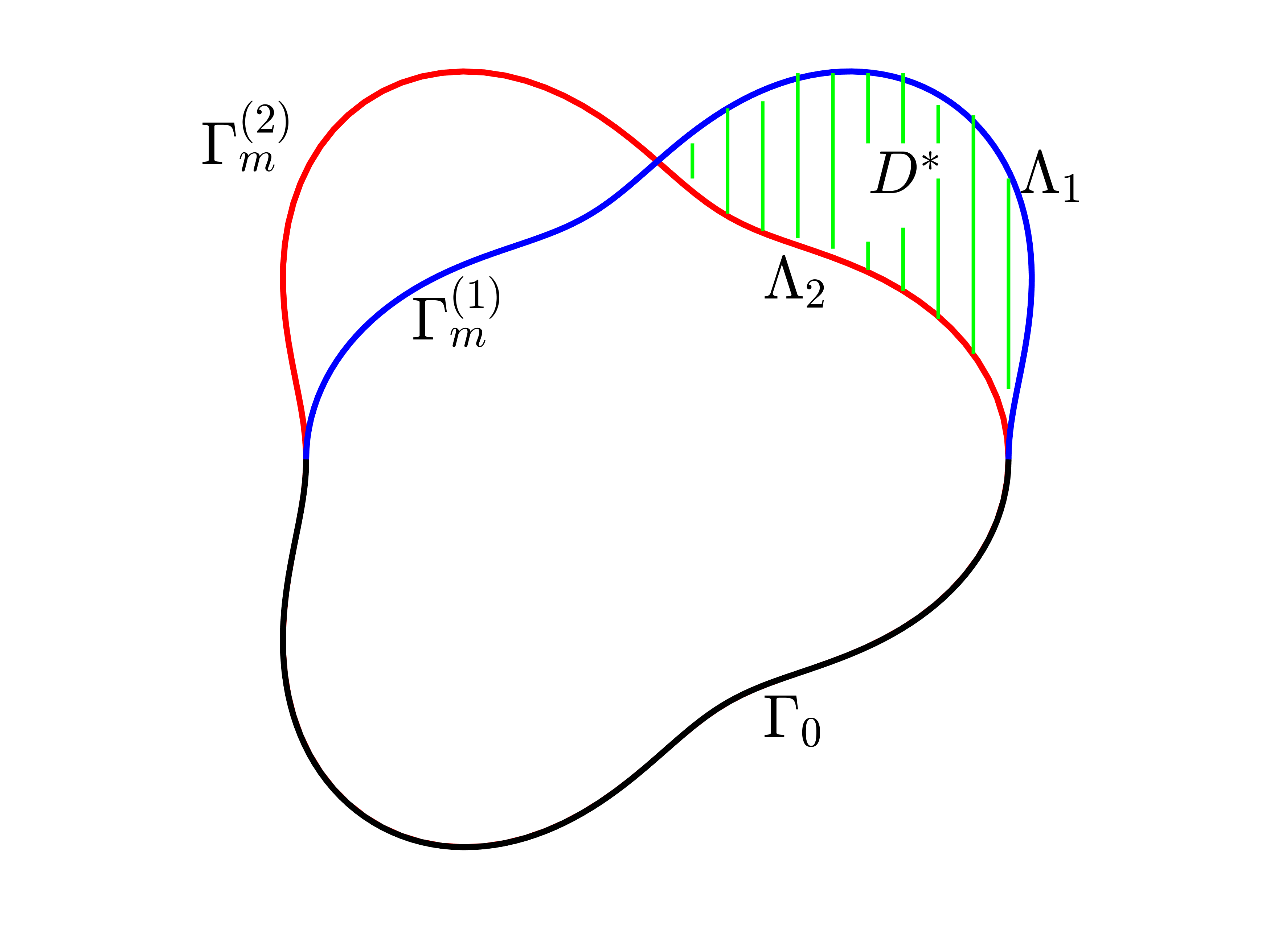}}%
			\caption{Sketch map}\label{fig1}
		\end{center}
	\end{figure}
	Since $\Lambda_1\subset\Gamma_m^{(1)}$, from the last equation (\ref{Th1-1}) we have
	\begin{equation}\label{Th1-2}
		T_{n_1}\bm u_1(\bm x, \omega)=-{\rm i}\omega\chi_1\bm u_1(\bm x, \omega).
	\end{equation}
	Since $\Lambda_2\subset\Gamma_m^{(2)}$, from the last equation (\ref{Th1-1}) we also have
	\begin{equation}\label{Th1-3}
		T_{n_2}\bm u_2(\bm x, \omega)=-{\rm i}\omega\chi_2\bm u_2(\bm x, \omega).
	\end{equation}
	Since $\bm{u}_1(\bm x, \omega)=\bm{u}_2(\bm x, \omega)$ and $T_{n}\bm u_1(\bm x, \omega)=T_{n}\bm u_2(\bm x, \omega)$ on $\Gamma_0$, the uniqueness of the Cauchy problem on $D_1\cap\overline{D^*}$ for elliptic equations will give $\bm{u}_1(\bm x, \omega)=\bm{u}_2(\bm x, \omega)$ on $\Lambda_2$, and thus $T_{n}\bm u_1(\bm x, \omega)=T_{n}\bm u_2(\bm x, \omega)$ on $\Lambda_2$. Again with the uniqueness of the Cauchy problem on $D^*$ will give $\bm{u}_1(\bm x, \omega)=\bm{u}_2(\bm x, \omega)$ in $D^*$.
	By $\Lambda_2\in (\overline{D_1}\cap\overline{D_2})$, together with $\bm u_1(\bm x, \omega)=\bm u_2(\bm x, \omega)$, and then
	\begin{equation}\label{Th1-4}
		T_{n_2}\bm u_1(\bm x, \omega)=-{\rm i}\omega\chi_2\bm u_1(\bm x, \omega).
	\end{equation}

	Let $\bm\nu=-\bm n_2, \chi=-\chi_2$ on $\Lambda_2$ and $\bm\nu=\bm n_1, \chi=\chi_1$ on $\Lambda_1$, combining
	equations (\ref{Th1-2}) and (\ref{Th1-4}), we have
	\begin{equation}\label{Th1boundary}
		T_{\bm\nu}\bm u_1(\bm x, \omega)=-{\rm i}\omega\chi\bm u_1(\bm x, \omega),~~{\rm on} ~~\partial D^*.
	\end{equation}
	Since $D^*\subset\overline{D_1}$, we know that $u_1(\bm x, \omega)$ satisfy the Navier equation in $D^*$ with the boundary condition (\ref{Th1boundary}).
	The first Betti formula yields
	\begin{eqnarray*}
		\int_{D^*}\left(E(\bm{u}_1,\bm{\overline{u}}_1)-\rho\omega^2|\bm{u}_1(\bm x, \omega)|^2\right)\mathrm{d} x
		=\int_{\partial D^*}T_{\bm\nu}\bm u_1(\bm x, \omega)\cdot{\bm{\overline{u}}_1(\bm x, \omega)}\mathrm{d} s,
	\end{eqnarray*}
	where

	\begin{eqnarray}\label{Eu1u1}
		E(\bm{u}_1,\bm{\overline{u}}_1)&=&\nabla\bm{u}_1(\bm x, \omega)\bm{:}\mathcal {C}\bm{:}\nabla\bm{\overline{u}}_1(\bm x, \omega)\\\nonumber &=&
		\lambda {\rm div}{\bm{u}_1}~{\rm div}{\bm{\overline{u}}_1}+\mu\sum_{p,q=1}^{2}\frac{\partial
			u_p}{\partial x_q}\left(\overline{\frac{\partial u_p}{\partial x_q}+\frac{\partial u_q}{\partial x_p}}\right) \nonumber\\
		&=&{\lambda}|{\rm div}{\bm{u}_1}|^2+2{\mu}\left(\bigg|\frac{\partial u_1}{\partial x_1}\bigg|^2+\bigg|\frac{\partial u_2}{\partial x_2}\bigg|^2\right)+{\mu}\left(\bigg|\frac{\partial u_1}{\partial x_2}+\frac{\partial u_2}{\partial x_1}\bigg|^2\right)\nonumber.
	\end{eqnarray}
	Combining the boundary condition (\ref{Th1boundary}), we have
	\begin{eqnarray*}
		\int_{D^*}\left(E(\bm{u}_1,\bm{\overline{u}}_1)-\rho\omega^2|\bm{u}_1(\bm x, \omega)|^2\right)\mathrm{d} \bm x
		+{\rm i}\omega\chi\int_{\partial D^*}|\bm{u}_1(\bm x, \omega)|^2 \mathrm{d} s=0.
	\end{eqnarray*}
	Taking the real part in the above equation, we have
	\begin{eqnarray}\label{Th1-10}
		\int_{D^*}\left(E(\bm{u}_1,\bm{\overline{u}}_1)-\rho\omega^2|\bm{u}_1(\bm x, \omega)|^2\right)\mathrm{d} \bm x=0,~~~{\rm for ~all } ~~\omega\in(a,b).
	\end{eqnarray}
	The analytical dependance of $u_1(\bm x, \omega)$ on $\omega$ implies that equation (\ref{Th1-10}) is valid  for $\omega\in\mathbb{C}$.
	
	Now choosing $\omega=1+{\rm i}\delta$ we get $\omega^2=1-\delta^2+2{\rm i}\delta$, then the equation (\ref{Th1-10}) will be
	\begin{eqnarray}\label{Th1-11}
		\int_{D^*}\left(E(\bm{u}_1,\bm{\overline{u}}_1)-\rho(1-\delta^2)|\bm{u}_1(\bm x, \omega)|^2\right)\mathrm{d} \bm x+2{\rm i}\rho\delta\int_{D^*}|\bm{u}_1(\bm x, \omega)|^2\mathrm{d} \bm x=0,
	\end{eqnarray}
	which implies
	\begin{eqnarray}\label{Th1-12}
		\rho\delta\int_{D^*}|\bm{u}_1(\bm x, \omega)|^2\mathrm{d} \bm x=0, ~~~{\rm for ~all } ~~\omega=1+{\rm i}\delta,~\delta\in\mathbb{R}
	\end{eqnarray}
	by taking the imaginary part of the equation (\ref{Th1-11}).
	Again using the analytical dependance, we get
	\begin{eqnarray}\label{Th1-12}
		\bm{u}_1(\bm x, \omega)=0 ~~{\rm in }~D^*
	\end{eqnarray}
	for all $\omega\in\mathbb{C}$. By the uniqueness of the analytic
	continuation, thus $\bm{u}_1(\bm x, \omega)=0$ in $D_1$ for all $\omega\in\mathbb{C}$. This yields $$\bm{f}(\bm x, \omega)=\bm{u}_1(\bm x, \omega)|_{\Gamma_0}=0,$$  which contradicts the fact that $\bm{f}(\bm x, \omega)\neq0$. Hence $\Gamma_m^{(1)}=\Gamma_m^{(2)}=\Gamma_m$, i.e., $D_1=D_2=D$.
	
	Next we will check $\chi_1=\chi_2$. On $\Gamma_0$, we have
	$$\bm{u}_1(\bm x, \omega)=\bm{u}_2(\bm x, \omega)$$
	and
	$$T_{n}\bm u_1(\bm x, \omega)=T_{n}\bm u_2(\bm x, \omega),$$
	the uniqueness of the Cauchy problem for elliptic equations will give
	\begin{equation}
		\bm{u}_1(\bm x, \omega)=\bm{u}_2(\bm x, \omega)~~~{\rm in}~~ D.
	\end{equation}
	Together with $T_{n}\bm u_1(\bm x, \omega)=-{\rm i}\omega\chi_1\bm u_1(\bm x, \omega)$ and $T_{n}\bm u_2(\bm x, \omega)=-{\rm i}\omega\chi_2\bm u_2(\bm x, \omega)$ on $\Gamma_m$,  we have ${\rm i}\omega\chi_1\bm u_1(\bm x, \omega)={\rm i}\omega\chi_2\bm u_1(\bm x, \omega)$, i.e., ${\rm i}\omega(\chi_1-\chi_2)\bm u_1(\bm x, \omega)=0$ on $\Gamma_m$. If  $\bm u_1(\bm x, \omega)=0$ on any open subset $\Gamma\subset\Gamma_m$, we have $T_{n}\bm u_1(\bm x, \omega)=-{\rm i}\omega\chi_1\bm u_1(\bm x, \omega)=0$ on $\Gamma$.  The uniqueness of the Cauchy problem for elliptic equations will give $\bm u_1(\bm x, \omega)\equiv0$ in $D$. This is a contradiction to the behavior of $\bm{u}_1(\bm x, \omega)$ on  $\Gamma_0$ with $\bm f(\bm x,\omega)\neq0$. Thus we have $\chi_1-\chi_2=0$ on $\Gamma_m$, and the proof is completed.
\end{proof1}

\section{Reconstruction}
In this section, we are concerned with the numerical method to reconstruct both the shape and the impendance function from the measurement ${\bm h}|_{\Gamma_m}.$ The overall approach is to divide the reconstruction procedure into two core steps: In the first step, we reconstruct the unknown displacement ${\bm u}$ by solving a Cauchy problem; In the second step, a reconstruction method based on the Newton iteration method will be proposed.

\subsection{The Cauchy problem}\label{subsec: Cauchy}
We first solve the Cauchy problem \eqref{jingque}. Taking the ill-posedness of the inverse problem into consideration, we instead solve the following perturbed equation
\begin{equation}\label{raodong}
	\mathcal{N}\bm\varphi^\delta=\bm {h}^\delta.
\end{equation}
Here, $\bm {h}^\delta=(\bm {f}^\delta,\bm {t}^\delta)^\top\in
\textbf{{\rm L}}^2(\Gamma_0)\times \textbf{{\rm L}}^2(\Gamma_0)$ is
measured noisy data satisfying
$$
\|\bm {h}^\delta-\bm {h}\|\leq\delta\|\bm{h}\|:=\epsilon,
$$
where $\delta$ is the noise level.

Seeking for a regularized solution to equation (\ref{jingque}) is to solve the following equation:
\begin{equation}\label{zhengzehua}
	{\alpha} \bm\varphi^\delta_{{\alpha}}
	+\mathcal{N}^*\mathcal{N}\bm\varphi^\delta_{{\alpha}}=\mathcal{N}^*\bm{h}^\delta,\quad{\alpha}>0,\nonumber
\end{equation}
where the regularization parameter $\alpha>0$ is chosen by Morozov discrepancy principle \cite{Sun14,Sun17,Sun172}. To be specific, the regularization parameter $\alpha$ is determined by finding the zero of
$G(\alpha):=\|\mathcal{N}\bm\varphi-\bm
h^\delta\|^2-\epsilon^2$.

If the regularization parameter is fixed by $\alpha^*$, then regularized solution $\bm\varphi^\delta_{\alpha^*}$ can be given by
$\bm\varphi^\delta_{\alpha^*}=(\alpha^*I + \mathcal{N}^*\mathcal{N})^{-1}\mathcal{N}^*\bm h^\delta$.
Correspondingly, the approximate displacement $\bm {u}^\delta=(u_1^\delta,u_2^\delta)^\top$ can be given by
\begin{align}\label{eq: udelta}
\bm {u}^\delta(\bm x)=\mathcal{S}\bm\varphi^\delta_{\alpha^*}(\bm x).
\end{align}

\subsection{Newton-type method}
So far, we have transformed the nonlinear and ill-posed inverse problem into a well-posed nonlinear problem. In this subsection, we shall adopt the Newton-type method to solve this nonlinear problem.

Once the approximated displacement \eqref{eq: udelta} is obtained, we are now ready to propose the
Newton-type iteration scheme to simultaneously reconstruct both the boundary $\partial D$ and impedance function $\chi$. For a given function $\bm g$ ($\bm g$ can be zero) and the approximated
displacement $\bm {u}^\delta(\bm x)$, we can define the mapping operator $\mathcal{F}$ from the boundary contour $\gamma$ and impedance function $\chi$ to the given function $\bm g$ by
\begin{align}\label{eq: F}
\mathcal{F}(\gamma,\chi)=T_{\bm n}\bm u^\delta(\gamma)+{\rm i}\omega\chi\bm u^\delta(\gamma).
\end{align}
In addition, it should to seek the impedance boundary where the boundary curve $\gamma$ and impedance function $\chi$ satisfy
$$\mathcal{F}(\gamma,\chi)=\bm g.$$

Our reconstruction algorithm can be summarized as follows:

\begin{enumerate}
	\item 	Given the noisy Cauchy data pair ${\bm h}^\delta=({\bm f}^\delta, {\bm t}^\delta)$ on $\Gamma_0$, we solve the Cauchy problem \ref{subsec: Cauchy} to derive the approximate displacement \eqref{eq: udelta} and the corresponding traction
	\begin{align*}
		\bm {t}^\delta(\bm x)=T_n\mathcal{S}\varphi^\delta_{\alpha^*}(\bm x).
	\end{align*}
	\item  Let $\left\{\gamma_n, \chi_n: n=0,1,2...\right\}$ be the current approximate to the boundary and the impedance, respectively. We update the boundary curve $\gamma_n$ and the impedance function $\chi_n$ by the following two procedures
	\begin{align}\label{eq: 4_3}
		&\begin{cases}
			\mathcal{F}(\gamma_n,\chi_n)+\mathcal{F'}(\gamma_n,\chi_n)\triangle \gamma_n=\bm g\\
			\gamma_{n+1}=\gamma_n+\triangle\gamma_n
		 \end{cases}\\\label{eq: 4_4}
		&\begin{cases}
			\mathcal{F}(\gamma_{n+1},\chi_n)+i\omega\mathcal{M}_D(\gamma_{n+1})\triangle \chi_n=\bm g\\
			\chi_{n+1}=\chi_n+\triangle\chi_n
		\end{cases}
	\end{align}
	respectively.
	\item Once we have gotten $\triangle \gamma_{n}$ and $\triangle\chi_n$, we update $\gamma_{n}$ and $\chi_n$ to $\gamma_{n+1}$ and $\chi_{n+1}$, respectively.
\end{enumerate}

We would like to point out that, during the iteration \eqref{eq: 4_3} and \eqref{eq: 4_4}, the Fr\'{e}chet derivatives of $\bm {u}^\delta(\bm x)$ and $\bm {t}^\delta(\bm x)$ with respect to $\bm x$ can be calculated explicitly as
\begin{equation}
	\nabla\bm {u}^\delta(\bm x)=
	\left[
	\begin{matrix}
		\partial_{x_{1}}u_{1}^\delta & \partial_{x_{2}}u_{1}^\delta \\
		\partial_{x_{1}}u_{2}^\delta & \partial_{x_{2}}u_{2}^\delta
	\end{matrix}
	\right],\quad
	\nabla\bm {t}^\delta(\bm x)=
	\left[
	\begin{matrix}
		\partial_{x_{1}}t_{1}^\delta & \partial_{x_{2}}t_{1}^\delta \\
		\partial_{x_{1}}t_{2}^\delta & \partial_{x_{2}}t_{2}^\delta
	\end{matrix}
	\right],\nonumber
\end{equation}
where
\begin{equation}
	\begin{split}
		\partial_{x_{1}}u_{1}^\delta=\int_{\partial B}\left(\partial_{x_{1}}\mathbb{E}_{11}(\bm x,\bm y)\varphi^\delta_{\alpha^*,1}(\bm{y})+\partial_{x_{1}}\mathbb{E}_{12}(\bm x,\bm y)\varphi^\delta_{\alpha^*,2}(\bm{y})\right)\mathrm{d}s(\bm{y}), \\
		\partial_{x_{2}}u_{1}^\delta=\int_{\partial B}\left(\partial_{x_{2}}\mathbb{E}_{11}(\bm x,\bm y)\varphi^\delta_{\alpha^*,1}(\bm{y})+\partial_{x_{2}}\mathbb{E}_{12}(\bm x,\bm y)\varphi^\delta_{\alpha^*,2}(\bm{y})\right)\mathrm{d}s(\bm{y}), \\
		\partial_{x_{1}}u_{2}^\delta=\int_{\partial B}\left(\partial_{x_{1}}\mathbb{E}_{21}(\bm x,\bm y)\varphi^\delta_{\alpha^*,1}(\bm{y})+\partial_{x_{1}}\mathbb{E}_{22}(\bm x,\bm y)\varphi^\delta_{\alpha^*,2}(\bm{y})\right)\mathrm{d}s(\bm{y}), \\
		\partial_{x_{2}}u_{2}^\delta=\int_{\partial B}\left(\partial_{x_{2}}\mathbb{E}_{21}(\bm x,\bm y)\varphi^\delta_{\alpha^*,1}(\bm{y})+\partial_{x_{2}}\mathbb{E}_{22}(\bm x,\bm y)\varphi^\delta_{\alpha^*,2}(\bm{y})\right)\mathrm{d}s(\bm{y}),
	\end{split}
	\nonumber
\end{equation}
and
\begin{align}
	\begin{split}
		\partial_{x_{1}}t_{1}^\delta=
		\int_{\partial B}\partial_{x_{1}}\mathbb{T}_{11}(\bm x,\bm y)\varphi^\delta_{\alpha^*,1}(\bm{y})+\partial_{x_{1}}\mathbb{T}_{12}(\bm x,\bm y)\varphi^\delta_{\alpha^*,2}(\bm{y})ds(\bm{y}), \\
		\partial_{x_{2}}t_{1}^\delta=
		\int_{\partial B}\partial_{x_{2}}\mathbb{T}_{11}(\bm x,\bm y)\varphi^\delta_{\alpha^*,1}(\bm{y})+\partial_{x_{2}}\mathbb{T}_{12}(\bm x,\bm y)\varphi^\delta_{\alpha^*,2}(\bm{y})ds(\bm{y}), \\
		\partial_{x_{1}}t_{2}^\delta=
		\int_{\partial B}\partial_{x_{1}}\mathbb{T}_{21}(\bm x,\bm y)\varphi^\delta_{\alpha^*,1}(\bm{y})+\partial_{x_{1}}\mathbb{T}_{22}(\bm x,\bm y)\varphi^\delta_{\alpha^*,2}(\bm{y})ds(\bm{y}), \\
		\partial_{x_{2}}t_{2}^\delta=
		\int_{\partial B}\partial_{x_{2}}\mathbb{T}_{21}(\bm x,\bm y)\varphi^\delta_{\alpha^*,1}(\bm{y})+\partial_{x_{2}}\mathbb{T}_{22}(\bm x,\bm y)\varphi^\delta_{\alpha^*,2}(\bm{y})ds(\bm{y}),
	\end{split}
	\nonumber
\end{align}
with
\begin{eqnarray*}
	&~& \partial_{x_{\imath}}\mathbb{T}_{1\jmath}( \bm{x}, \bm{y})\\
	&=&\left[(\lambda+2\mu)\frac{\partial^2  \mathbb{E}_{1\jmath}( \bm{x}, \bm{y})}{\partial x_1\partial x_\imath}+\lambda\frac{\partial^2  \mathbb{E}_{2\jmath}( \bm{x}, \bm{y})}{\partial x_2\partial x_\imath}\right]n_1(\bm x)+\mu\left(\frac{\partial^2 \mathbb{E}_{1\jmath}( \bm{x}, \bm{y})}{\partial x_2\partial x_\imath}+\frac{\partial^2  \mathbb{E}_{2\jmath}( \bm{x}, \bm{y})}{\partial x_1\partial x_\imath}\right)n_2(\bm x)\\
	&+&\left[(\lambda+2\mu)\frac{\partial  \mathbb{E}_{1\jmath}( \bm{x}, \bm{y})}{\partial x_1}+\lambda\frac{\partial  \mathbb{E}_{2\jmath}( \bm{x}, \bm{y})}{\partial x_2}\right]\frac{\partial n_1(\bm{x})}{\partial x_\imath}+\mu\left(\frac{\partial \mathbb{E}_{1\jmath}( \bm{x}, \bm{y})}{\partial x_2}+\frac{\partial  \mathbb{E}_{2\jmath}( \bm{x}, \bm{y})}{\partial x_1}\right)\frac{\partial n_2(\bm{x})}{\partial x_\imath},
\end{eqnarray*}
\begin{eqnarray*}
	&~& \partial_{x_{\imath}}\mathbb{T}_{2\jmath}( \bm{x}, \bm{y})\\
	&=&\mu\left(\frac{\partial^2 \mathbb{E}_{1\jmath}( \bm{x}, \bm{y})}{\partial x_2\partial x_\imath}+\frac{\partial^2  \mathbb{E}_{2\jmath}( \bm{x}, \bm{y})}{\partial x_1\partial x_\imath}\right)n_1(\bm x)+\left[\lambda\frac{\partial^2  \mathbb{E}_{1\jmath}( \bm{x}, \bm{y})}{\partial x_1\partial x_\imath}+(\lambda+2\mu)\frac{\partial^2  \mathbb{E}_{2\jmath}( \bm{x}, \bm{y})}{\partial x_2\partial x_\imath}\right]n_2(\bm x)\\
	&+&\mu\left(\frac{\partial \mathbb{E}_{1\jmath}( \bm{x}, \bm{y})}{\partial x_2}+\frac{\partial  \mathbb{E}_{2\jmath}( \bm{x}, \bm{y})}{\partial x_1}\right)\frac{\partial n_1(\bm{x})}{\partial x_\imath}+\left[\lambda\frac{\partial  \mathbb{E}_{1\jmath}( \bm{x}, \bm{y})}{\partial x_1}+(\lambda+2\mu)\frac{\partial  \mathbb{E}_{2\jmath}( \bm{x}, \bm{y})}{\partial x_2}\right]\frac{\partial n_2(\bm{x})}{\partial x_\imath}.
\end{eqnarray*}

To stop the iteration procedure, we define the following relative error
\begin{equation}\label{eq: error}
	E_{n}=\frac{\|\Delta \gamma_{n}\|_{L^{2}}}{\|\gamma_{n-1}\|_{L^{2}}}+\frac{\|\Delta \chi_{n}\|_{L^{2}}}{\|\chi_{n-1}\|_{L^{2}}},
\end{equation}
and choose some proper constant $\varepsilon >0$. When $E_{n}<\varepsilon$, we stop the iteration and take the current result as the reconstruction.

\subsection{Parameterization}
For simplicity, we assume that the boundary is starlike and can be represented in the parametric form
$$\partial D=\left\{r(\vartheta)(\cos\vartheta,\sin\vartheta),~\vartheta\in[0,2\pi)\right\}$$
where $r\in C^2\left([0,2\pi),\mathbb{R}_+\right)$ is a positive, twice continuously differentiable and $2\pi$-periodic function.
Take $$\Gamma_0=\{r(\vartheta)(\cos\vartheta,\sin\vartheta),~\vartheta\in[0,\pi)\}$$ as the known part coincides with the measured Cauchy data and $$\Gamma_m=\{r(\vartheta)(\cos\vartheta,\sin\vartheta),~\vartheta\in[\pi,2\pi)\}$$ as the missing part of the boundary with the impedance boundary condition.

To numerically approximate $r(\vartheta)$, we assume that $r(\vartheta)$ is represented as the trigonometric polynomials of degree less than or equal to
$N$, namely
\begin{eqnarray}\label{rounumber}
	r(\vartheta)=a_0+\sum_{\jmath=1}^N(a_{\jmath}\cos \jmath\vartheta+b_{\jmath}\sin \jmath\vartheta).
\end{eqnarray}

In the $n$-th ($n=1,2,\cdots$) iterative step, the radial function $r_n$ will be updated to $r_{n+1}$ by $\Delta r_n$. This process will be carried on by updating the Fourier coefficients $\bm a=(a_0,a_1,...,a_N,b_1,...,b_N)^\top$. Denote the Fourier coefficients at the $n$-th step by $\bm a^{(n)}=\left(a_0^{(n)},a_1^{(n)},...,a_N^{(n)},b_1^{(n)},...,b_N^{(n)}\right)^\top$. For notation simplify, we further denote by
$$
\mathcal{Q}^{(n)}=\bm t^\delta(B\bm a^{(n)}\hat{x})+{\rm i}\omega\chi_n\bm u^\delta(B\bm a^{(n)}\hat{x}),\quad n=1,2,\cdots,
$$
with $\hat{x}=(\cos\vartheta,\sin\vartheta)^\top$ and $B=(1,\cos \vartheta,...,\cos N\vartheta,\sin\vartheta,...,\sin N\vartheta)$. Then the Fourier coefficients in iterative process \eqref{eq: 4_3} will be updated by $\bm a^{(n+1)}=\bm a^{(n)}+\triangle\bm a^{(n)}$ by the following procedure
\begin{eqnarray*}
	\begin{cases}
		\mathcal{Q}^{(n)}+\nabla\mathcal{Q}^{(n)}\left(B\Delta\bm a^{(n)}\right)\hat{x}=\bm g(B\bm a^{(n)}\hat{x})\\
		\bm a^{(n+1)}=\bm a^{(n)}+\triangle \bm a^{(n)}
	\end{cases},\quad n=1,2,\cdots.
\end{eqnarray*}
Whilst the update impedance function $\chi_{n+1}$ in iterative process \eqref{eq: 4_4} can be reformulated as
\begin{eqnarray}
	\begin{cases}
		\mathcal{Q}^{(n+1)}+{\rm i}\omega\bm u^\delta(B\bm a^{(n+1)}\hat{x})\triangle \chi_n=\bm g(B\bm a^{(n+1)}\hat{x})\\
		\chi_{n+1}=\chi_n+\triangle\chi_n
	\end{cases},\quad n=1,2,\cdots.
\end{eqnarray}
Correspondingly, the stopping rule defined by \eqref{eq: error} is replaced by the following relative error
\begin{equation}
	E_{n}=\frac{\|\Delta \bm a^{(n)}\|_{L^{2}}}{\|\bm a^{(n-1)}\|_{L^{2}}}+\frac{\|\Delta \chi_{n}\|_{L^{2}}}{\|\chi_{n-1}\|_{L^{2}}}.
\end{equation}
For the chosen constant $\varepsilon >0$. When $E_{n}<\varepsilon$, we stop the iteration and take the current result as the reconstruction.

\section{Numerical Examples}
In this section, several numerical examples will be presented to illustrate the effectiveness of the proposed method. The Lem\'{e} constants $\lambda$ and $\mu$ are chosen to be $\lambda=\mu=1$. The inside of the elastic obstacle is assumed to be filled with
a homogeneous and isotropic elastic medium with a unit mass, i.e. $\rho=1$.  The initial guesses for the boundary and the impedance function are chosen to be a half circle, and a constant function, respectively. To approximate the radial function numerically, we take $N=8$ in equation (\ref{rounumber}). The regularization parameter is determined by the Morozov discrepancy principle.

In reality, we know all the information about the known part $\Gamma_0$ together with the Cauchy data pair $(\bm f, \bm t)$. Thus we can easily derive the equation $T_n\bm u+i\omega\chi\bm u=\bm g$ on $\Gamma_0$. Further, we can carry on iterative procedure on the whole boundary $\partial D$. The initial guess of the boundary is a circle. The results about the known part can provide a reference result for the reconstruction. To stop the iteration, the error tolerance $\varepsilon$ is chosen to be $10^{-5}.$

\begin{example}\label{EX1}
	In this example, we choose $D$ to be a bean-shaped domain, whose boundary $\partial D$ can be parameterized by
	$$\bm z_1(\vartheta)=\frac{1+0.8\cos\vartheta+0.2\sin2\vartheta}{1+0.7\cos\vartheta}(\cos\vartheta,\sin\vartheta).$$
	The parameterized surface impedance $\chi$ is given by $$\chi_1(\vartheta)=\sin^4\vartheta+1, \vartheta\in[\pi,2\pi).$$ The components $u_p$ and $u_s$ of the displacement $\bm u$ inside the elastic obstacle are given by $$u_p(\bm x)=H_0^{(1)}(\kappa_p|\bm x-\bm y_0|),~~~u_s(\bm x)=H_0^{(1)}(\kappa_s|\bm x-\bm y_0|),~~~\bm x \in \overline{D},$$ respectively,
	where $\bm y_0=(1,0)$, i.e. $\bm u(\bm x)=\nabla u_p(\bm x)+\nabla^\bot \bm u_s(\bm x)$.  Fix the frequency to be $\omega=3$.
	The virtual boundary $\partial B$ is chosen to be a circle centered at the origin with radius $4$. Take the initial guess for the boundary to be a circle centered at the origin with radius $0.3$. Figure \ref{figEX1:1} shows the reconstruction of the missing boundary and the impedance function under different noise levels $\delta\in\{0,1\%,5\%\}$. We can see from Figure \ref{figEX1:1} that the proposed method exhibits satisfactorily in reconstructing the boundary and the reconstruction improves as the noise decreases.
\end{example}
\begin{figure}[htp]
	\begin{center}
		\subfigure[Noise free]{\includegraphics[width=0.33\linewidth]{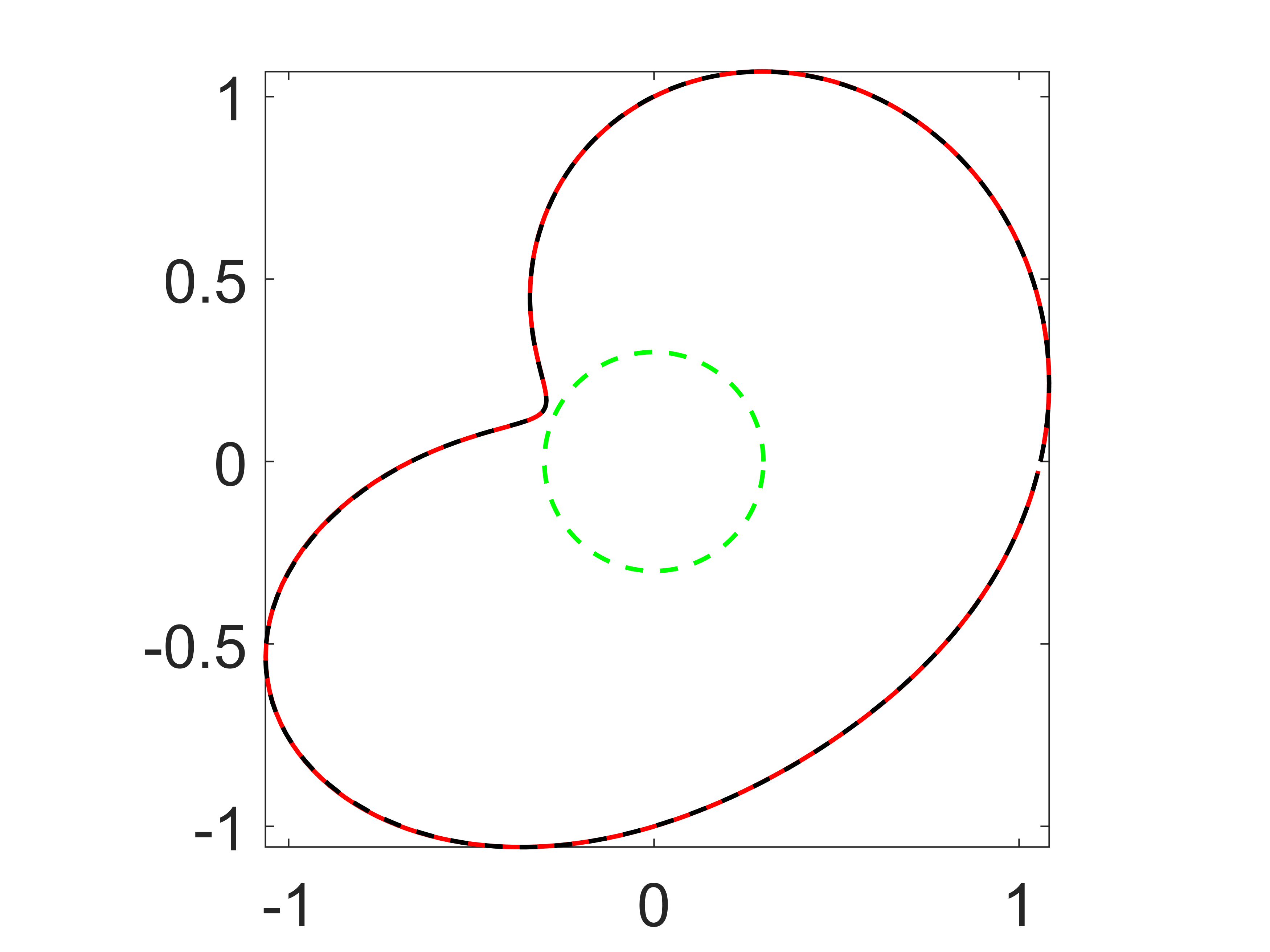}}%
		\subfigure[Noise 1$\%$]{\includegraphics[width=0.33\linewidth]{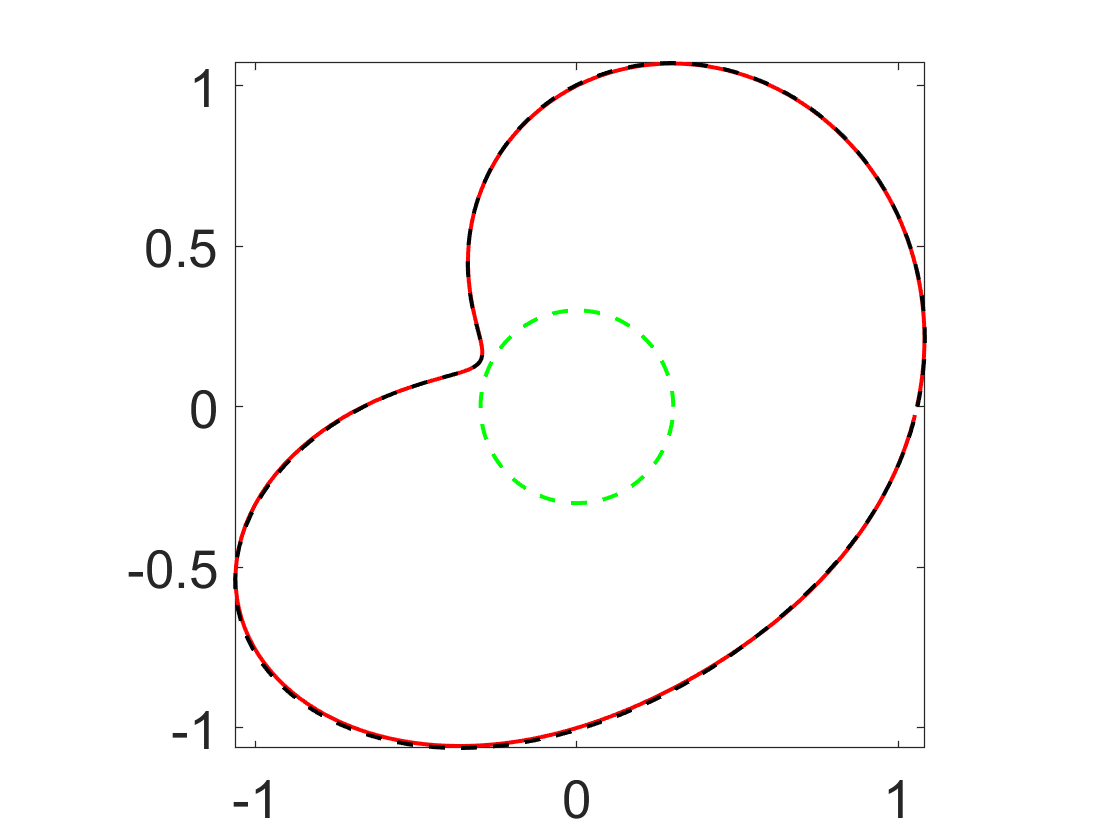}}%
		\subfigure[Noise 5$\%$]{\includegraphics[width=0.33\linewidth]{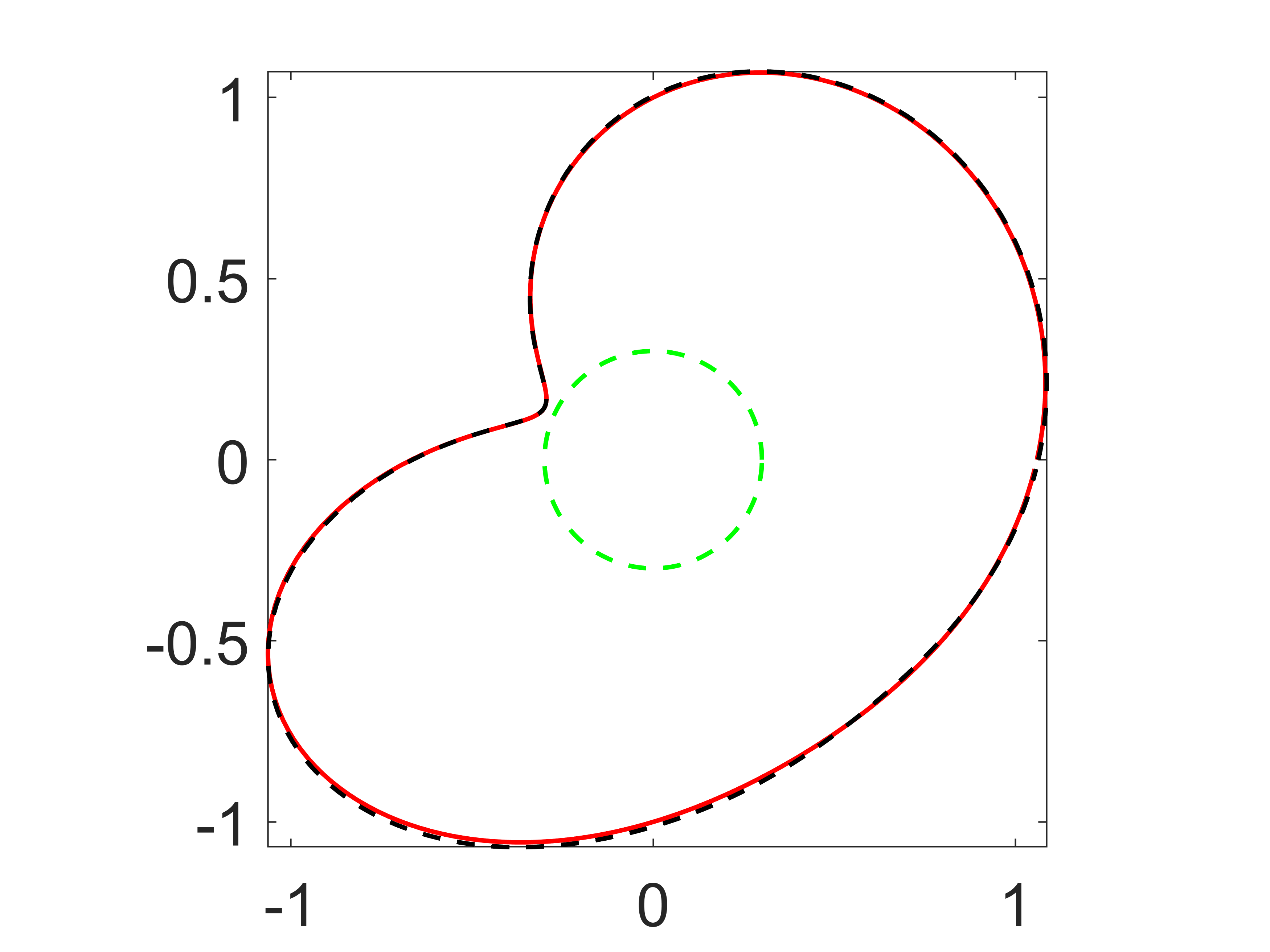}}\\%
		\subfigure[Noise free]{\includegraphics[width=0.33\linewidth]{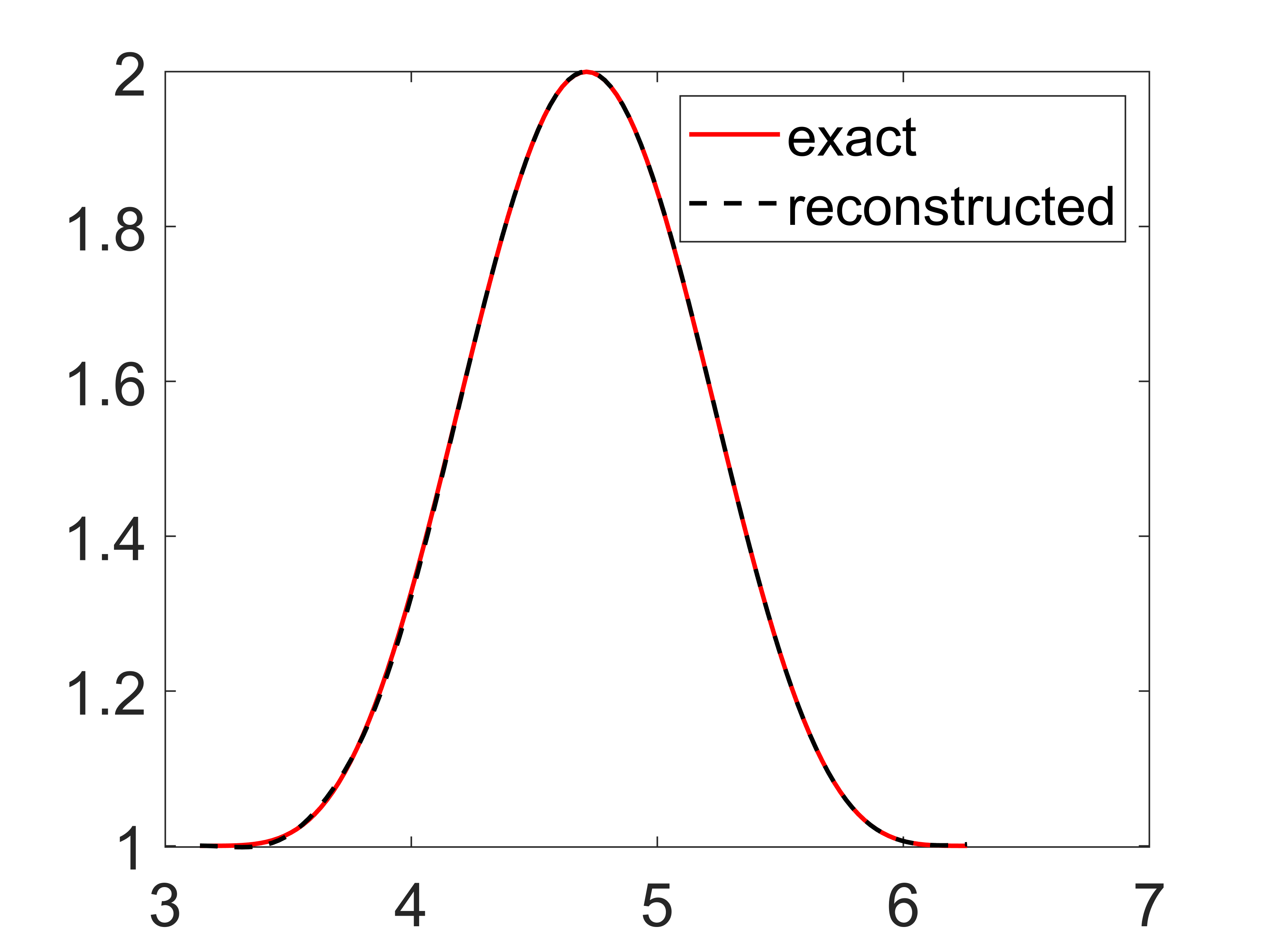}}%
		\subfigure[Noise 1$\%$]{\includegraphics[width=0.33\linewidth]{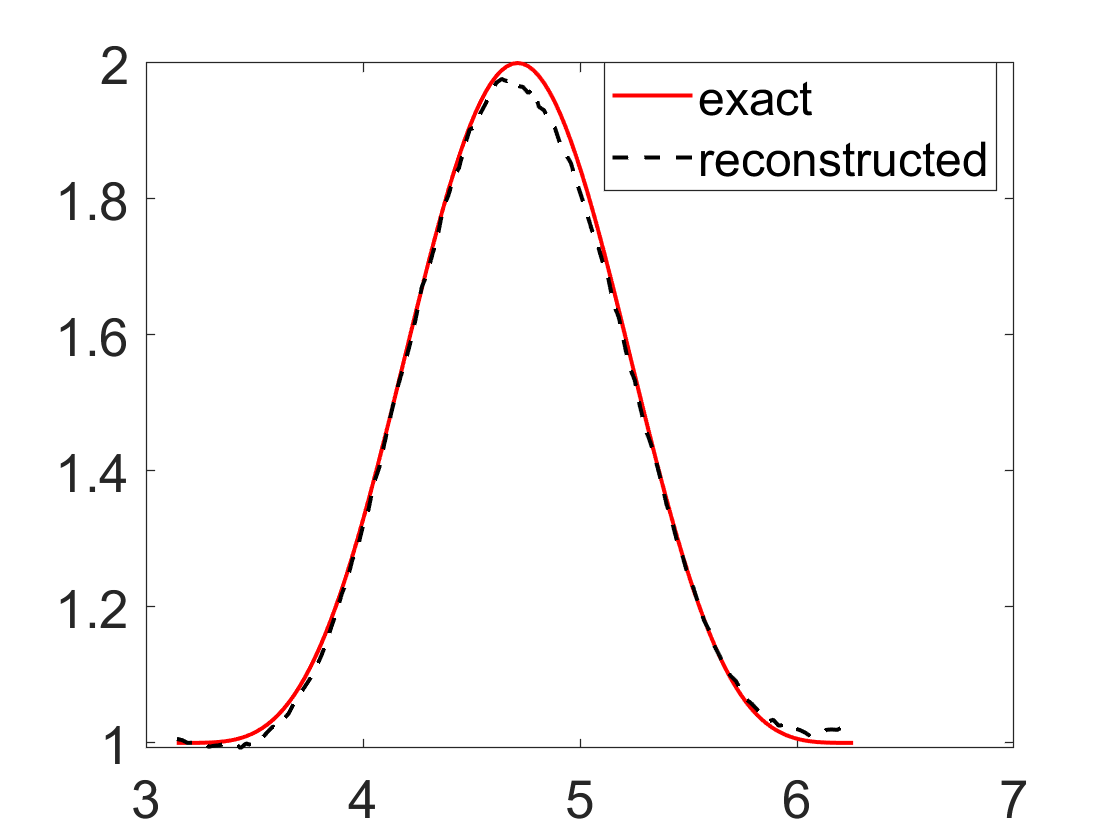}}%
		\subfigure[Noise 5$\%$]{\includegraphics[width=0.33\linewidth]{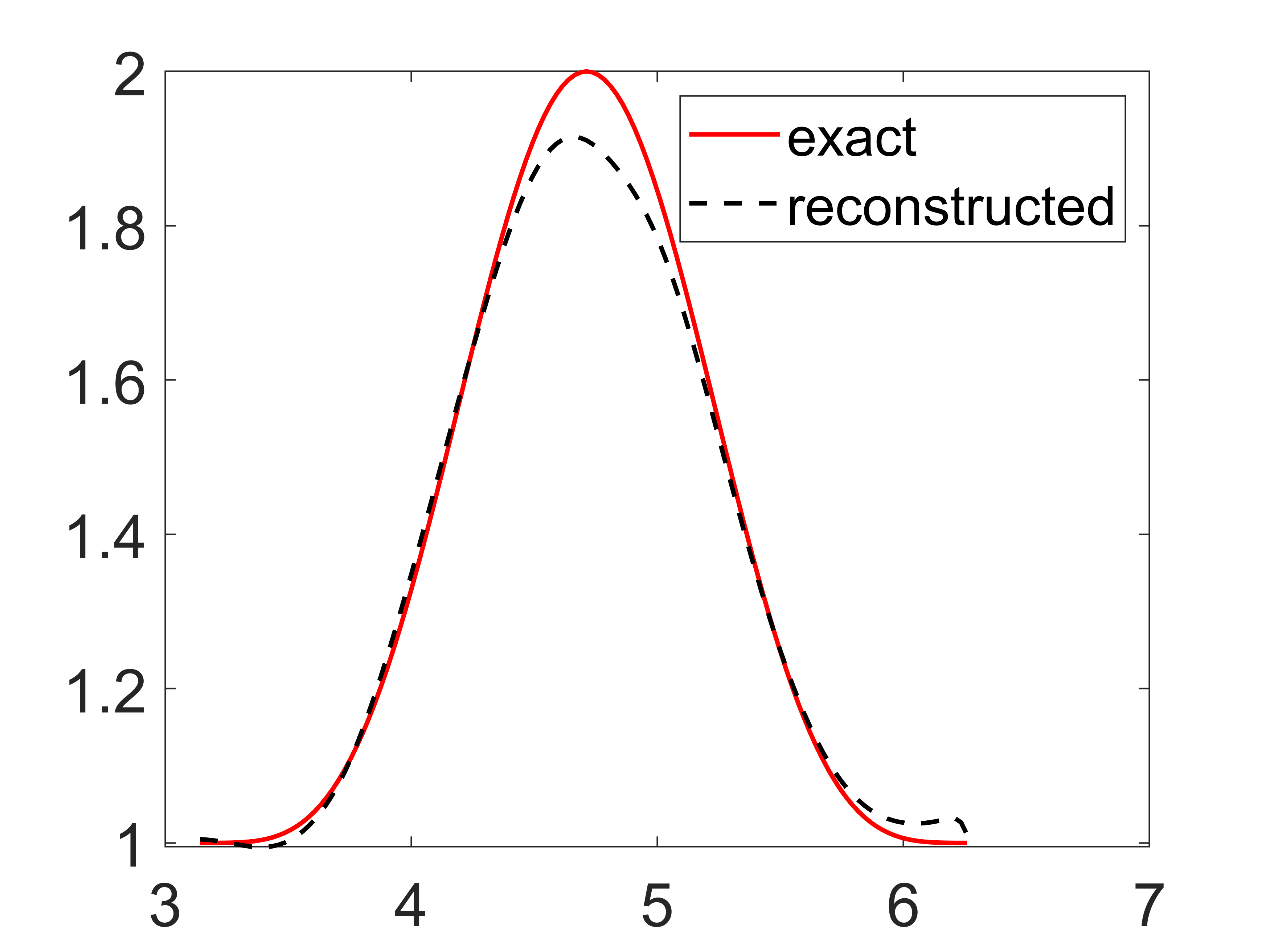}}\\%
		\caption{Reconstruction of the bean-shaped domain  and the impedance function under different noise levels $\delta\in\{0,1\%,5\%\}$.}\label{figEX1:1}
	\end{center}
\end{figure}

\begin{example}\label{EX2}
	In the second example, we consider the case in which the components $u_p$ and $u_s$ of the displacement $\bm u$ inside the elastic obstacle are respectively generated by $$u_p(\bm x)=\frac{{\rm i}}{4}H_0^{(1)}(\kappa_p|\bm x-\bm y_0|),~~~u_s(\bm x)=\frac{{\rm i}}{4}H_0^{(1)}(\kappa_s|\bm x-\bm y_0|),~~~\bm x \in \overline{D},$$
	i.e. $\bm u(\bm x)=\nabla u_p(\bm x)+\nabla^\bot \bm u_s(\bm x)$.
	Take the virtual boundary $\partial B$ to be a circle centered at the origin with radius $4$. The initial guess of the boundary is chosen to be a circle centered at the origin with radius $0.3$.

We first consider the case in which $D$ is a peanut-shaped domain with whose boundary $\partial D$ parameterized by
$$\bm z_2(\vartheta)=0.5\left(4\cos^2\vartheta+\sin^2\vartheta\right)^{\frac{1}{2}}(\cos\vartheta,\sin\vartheta).$$
Fix the frequency to be $\omega=5$, and take $\bm y_0=(4,-9)$. The parameterized surface impedance $\chi$ is described by $$\chi_2(\vartheta)=\sin^4\vartheta, \vartheta\in[\pi,2\pi).$$

In Figure \ref{figE2:1}, we display the reconstruction of the missing boundary and the impedance function under different noise levels $\delta\in\{0,1\%,5\%\}$. We can see from Figure \ref{figE2:1} that the reconstruction for the impedance function $\chi_2$ is insensitive to the noise level and the reconstruction for the peanut-shaped boundary is of high accuracy.

Further, we take $D$ to be a starfish-shaped domain, whose boundary $\partial D$ is described by
$$\bm z_3(\vartheta)=\left(1+0.2\cos5\vartheta\right)(\cos\vartheta,\sin\vartheta).$$
Take $\omega=3$, $\bm y_0=(4,9)$. The surface impedance is chosen to be a constant function $$\chi_3(\vartheta)=1, \vartheta\in[\pi,2\pi).$$
We plot the reconstruction for the starfish together with the impedance in Figure \ref{figE2:2}. We can see from Figure \ref{figE2:2} that the boundary curve is accurately reconstructed and the impedance function can be well-reconstructed when the exact data is utilized in the reconstruction. When there is some noise involved, the reconstruction may be influenced.
\end{example}

\begin{figure}[htp]
	\begin{center}
		\subfigure[Noise free]{\includegraphics[angle=0,
			width=0.33\linewidth]{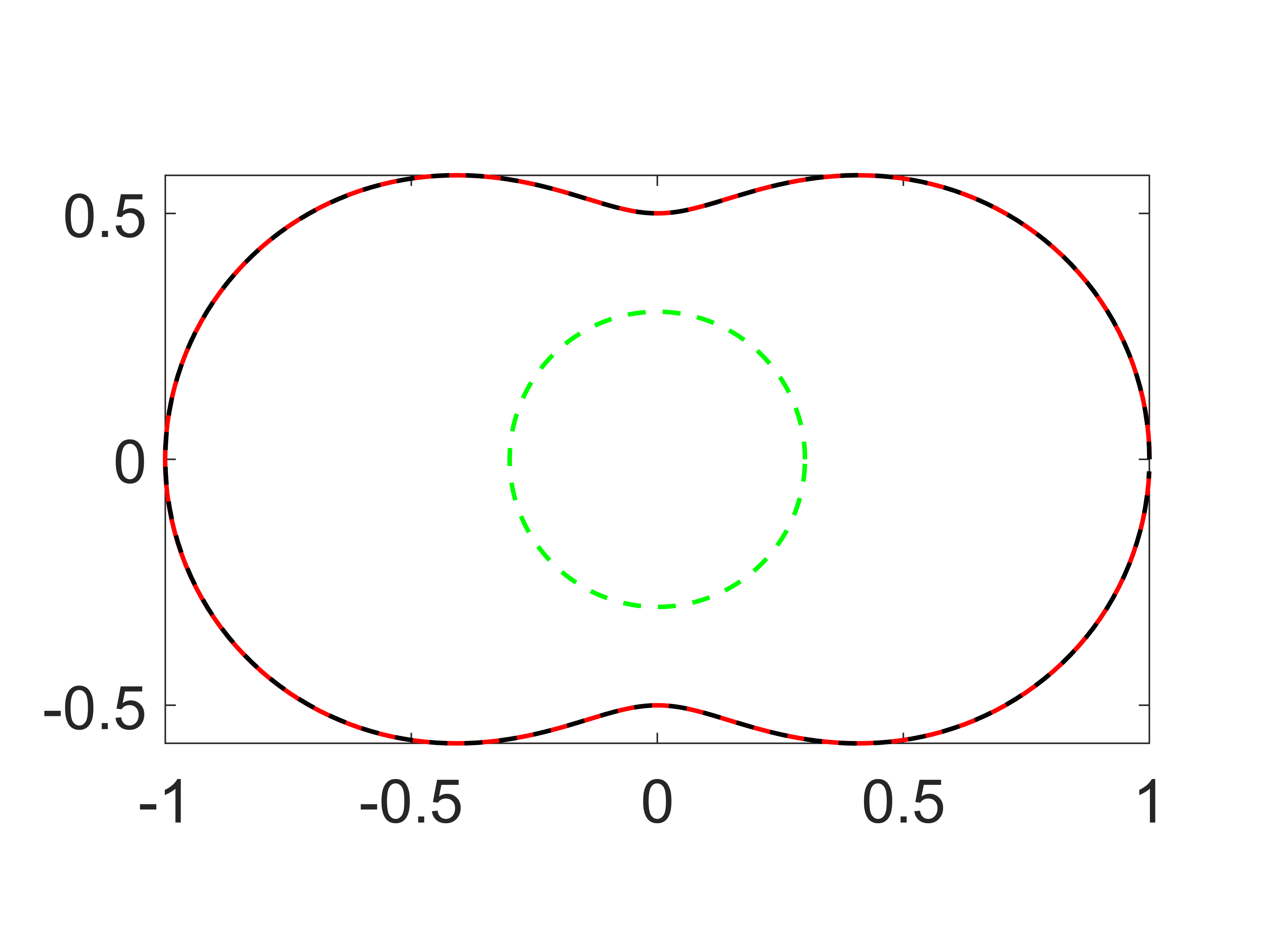}}%
		\subfigure[Noise 1$\%$]{\includegraphics[angle=0,
			width=0.33\linewidth]{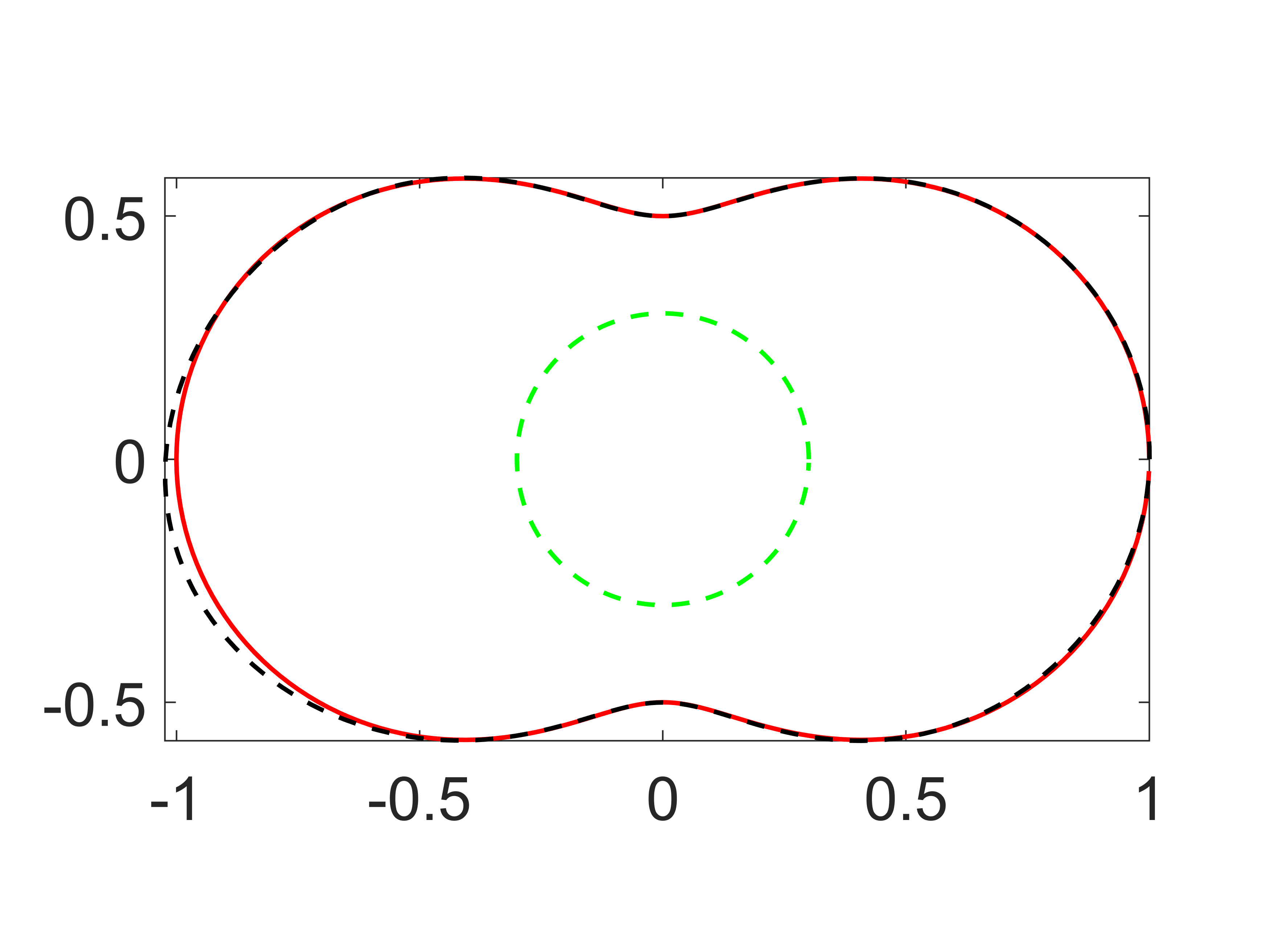}}%
		\subfigure[Noise 5$\%$]{\includegraphics[angle=0,
			width=0.33\linewidth]{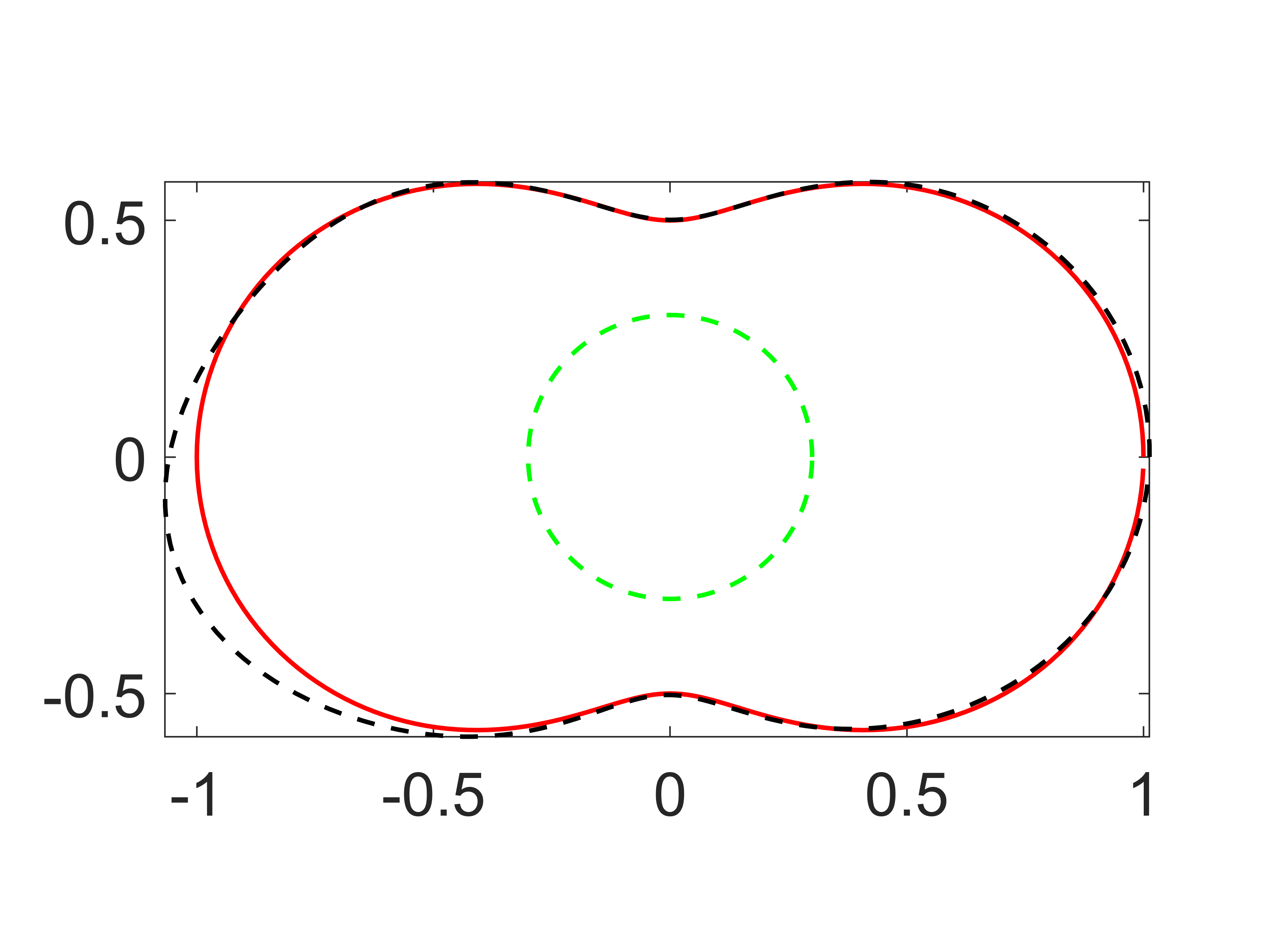}}\\%
		\subfigure[Noise free]{\includegraphics[angle=0,
			width=0.33\linewidth]{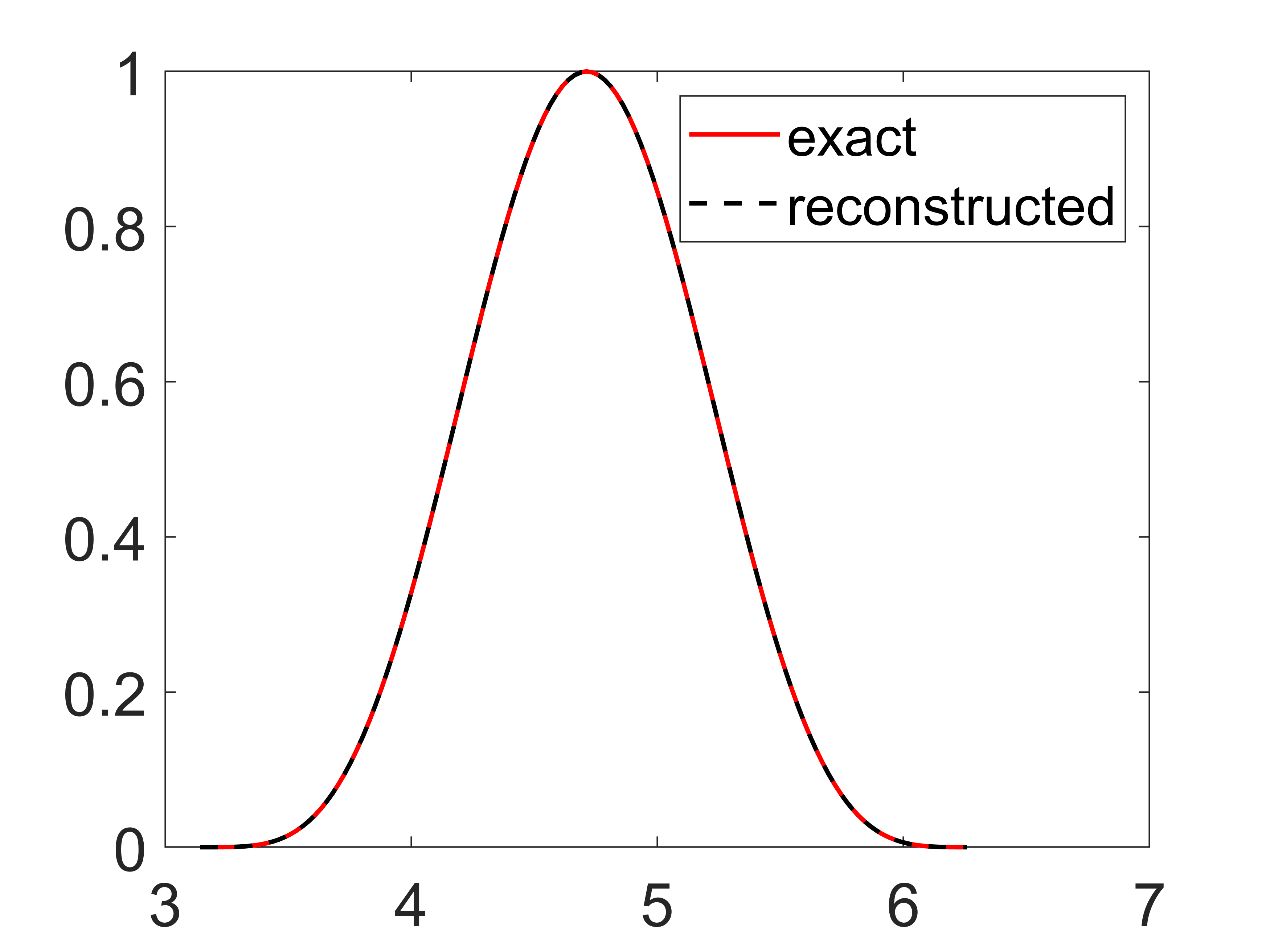}}%
		\subfigure[Noise 1$\%$]{\includegraphics[angle=0,
			width=0.33\linewidth]{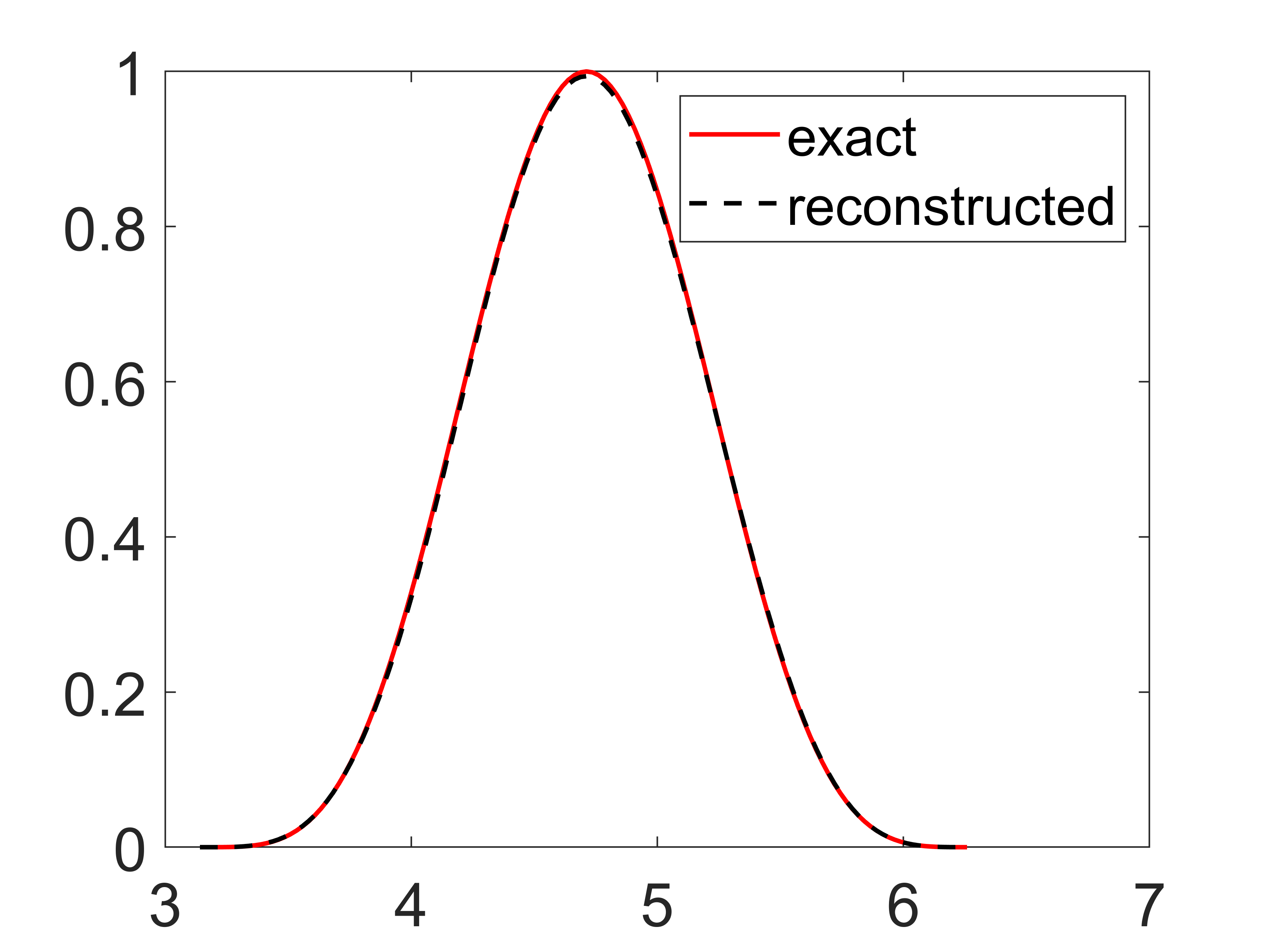}}%
		\subfigure[Noise 5$\%$]{\includegraphics[angle=0,
			width=0.33\linewidth]{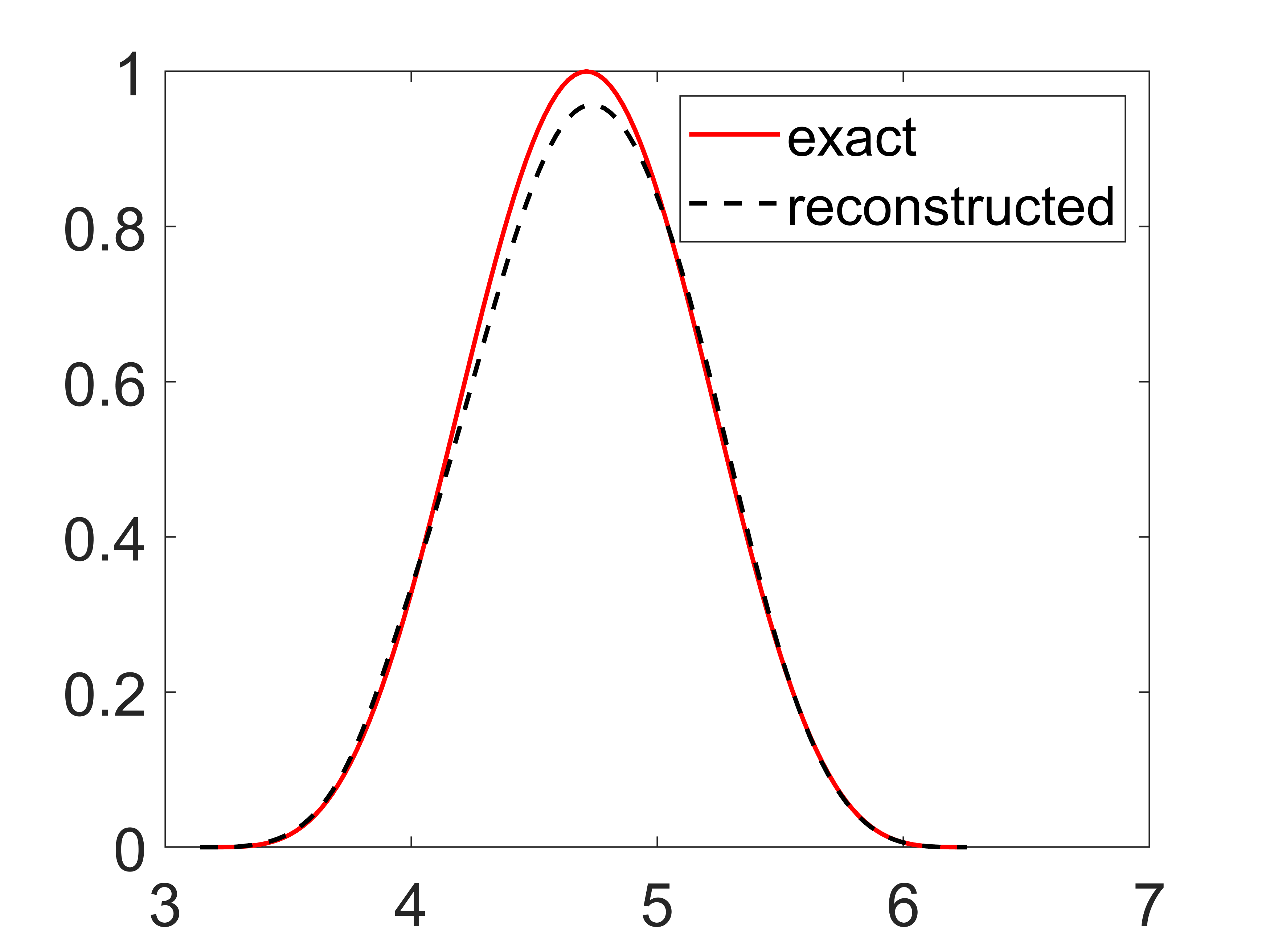}}%
		\caption{Reconstruction of peanut-shaped domain (the first row) and the impedance function (the second row) under different noise levels $\delta\in\{0,1\%,5\%\}$ in Example \ref{EX2}.}\label{figE2:1}
	\end{center}
\end{figure}

\begin{figure}[htp]
	\begin{center}
		\subfigure[Noise free]{\includegraphics[angle=0,
			width=0.33\linewidth]{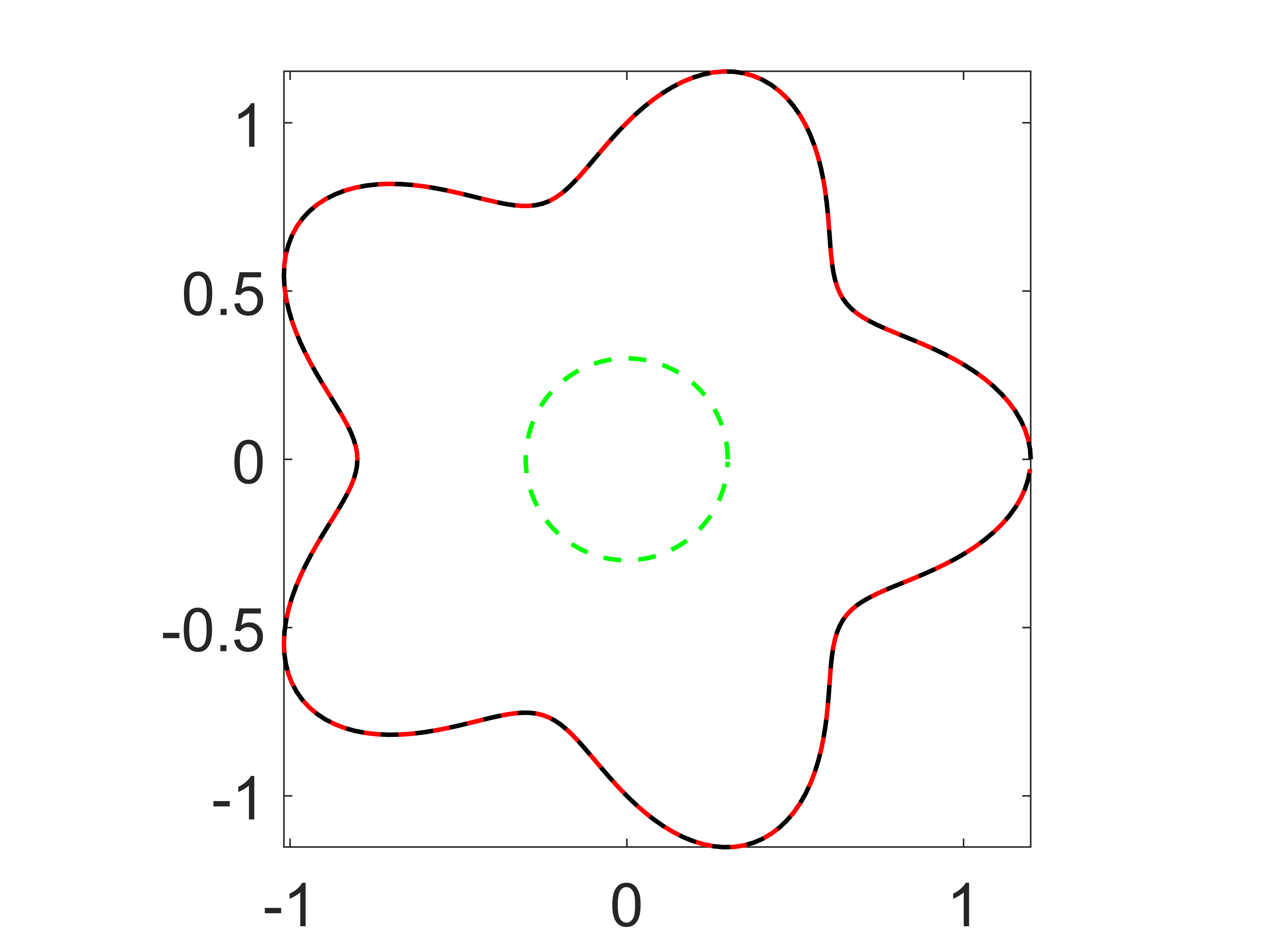}}%
		\subfigure[Noise 1$\%$]{\includegraphics[angle=0,
			width=0.33\linewidth]{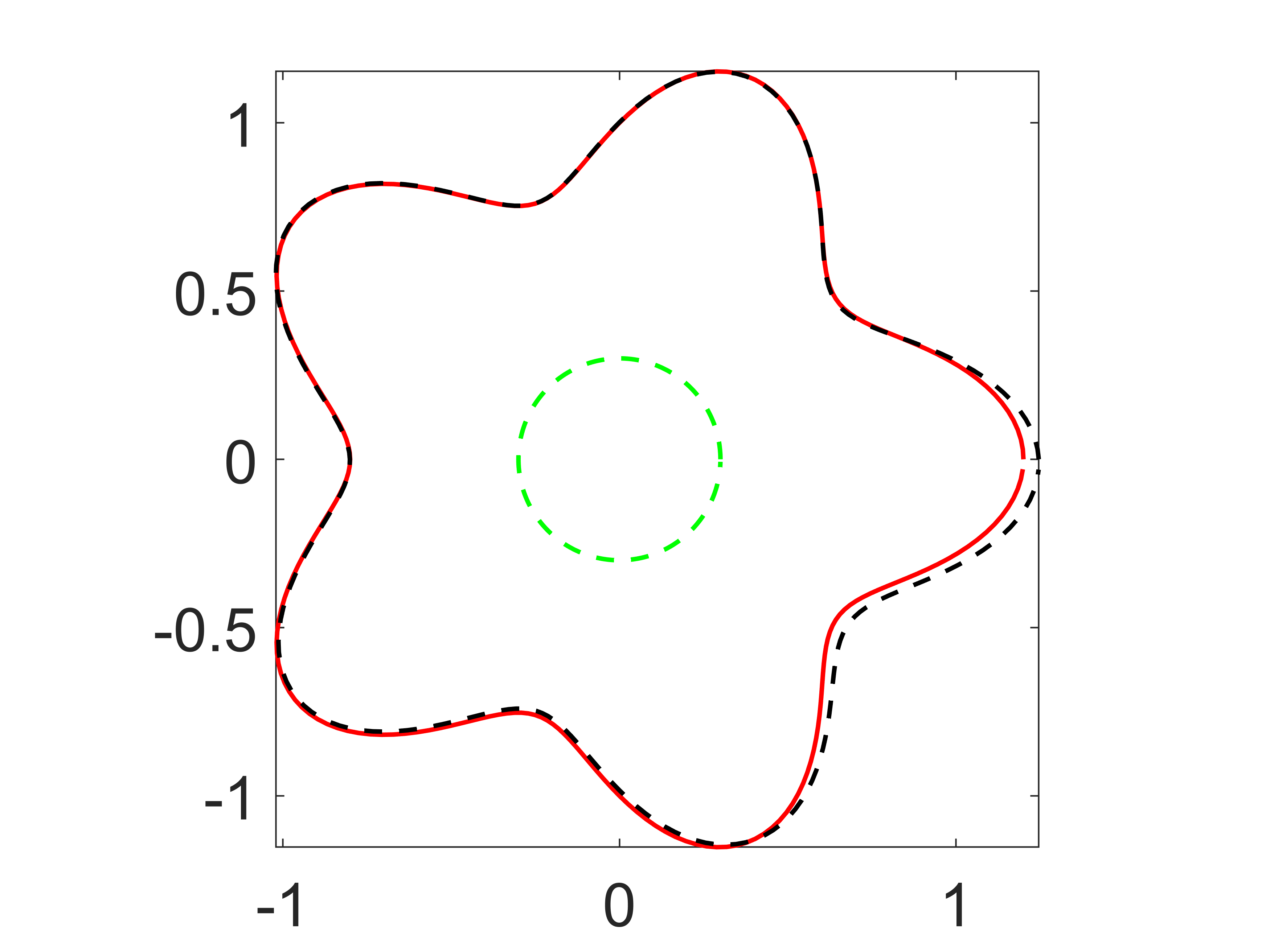}}%
		\subfigure[Noise 5$\%$]{\includegraphics[angle=0,
			width=0.33\linewidth]{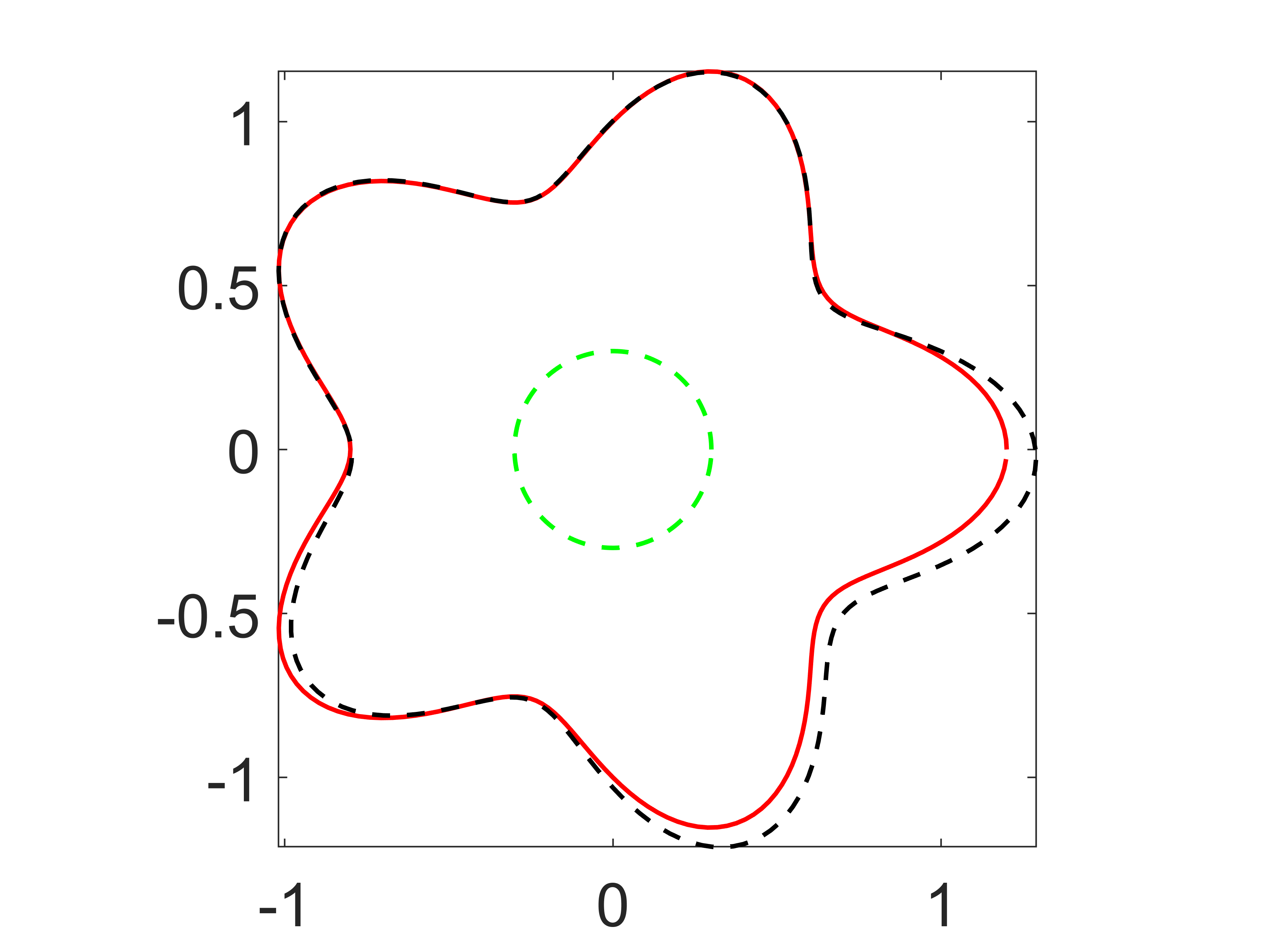}}\\%
		\subfigure[Noise free]{\includegraphics[angle=0,
			width=0.33\linewidth]{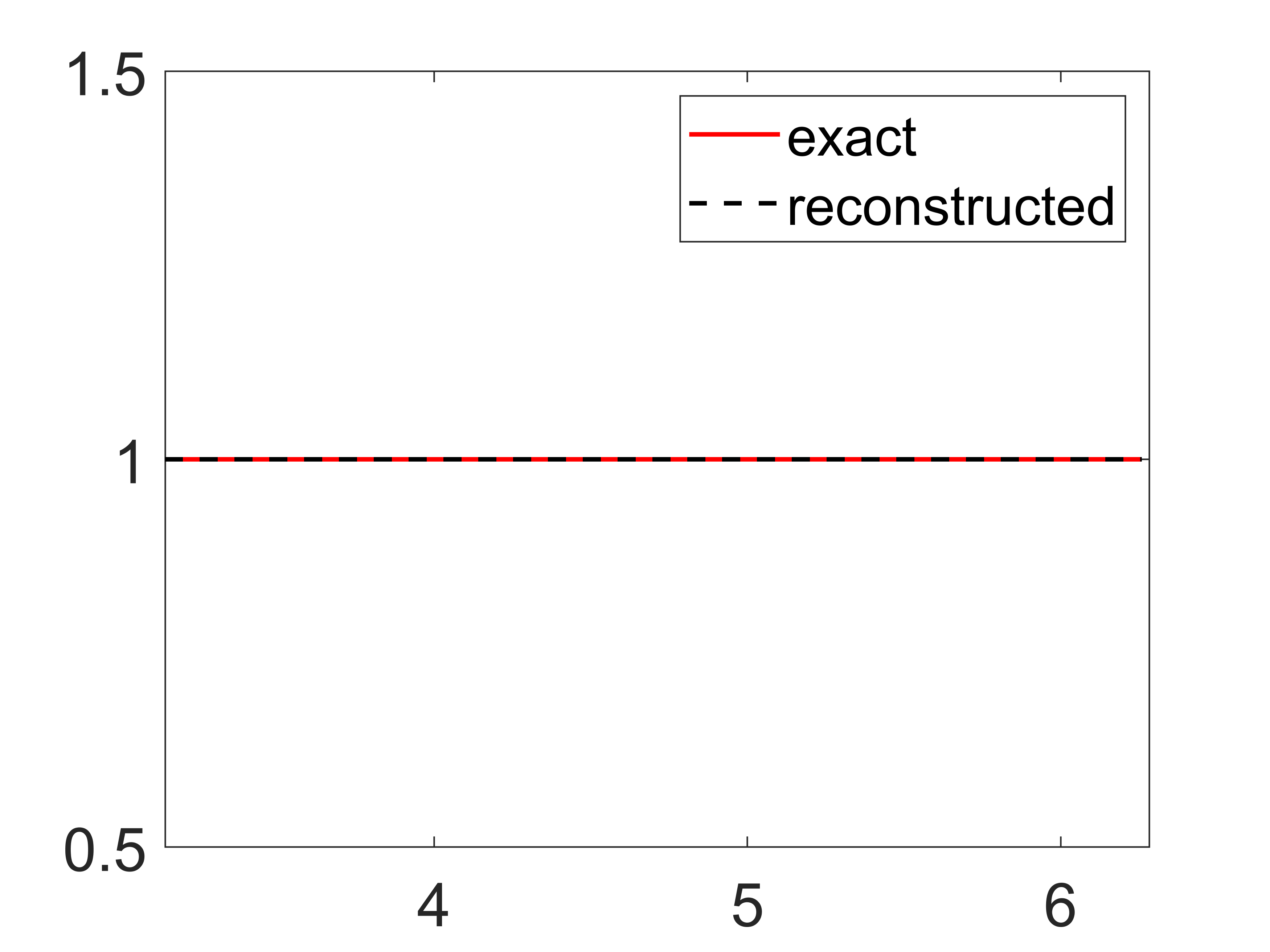}}%
		\subfigure[Noise 1$\%$]{\includegraphics[angle=0,
			width=0.33\linewidth]{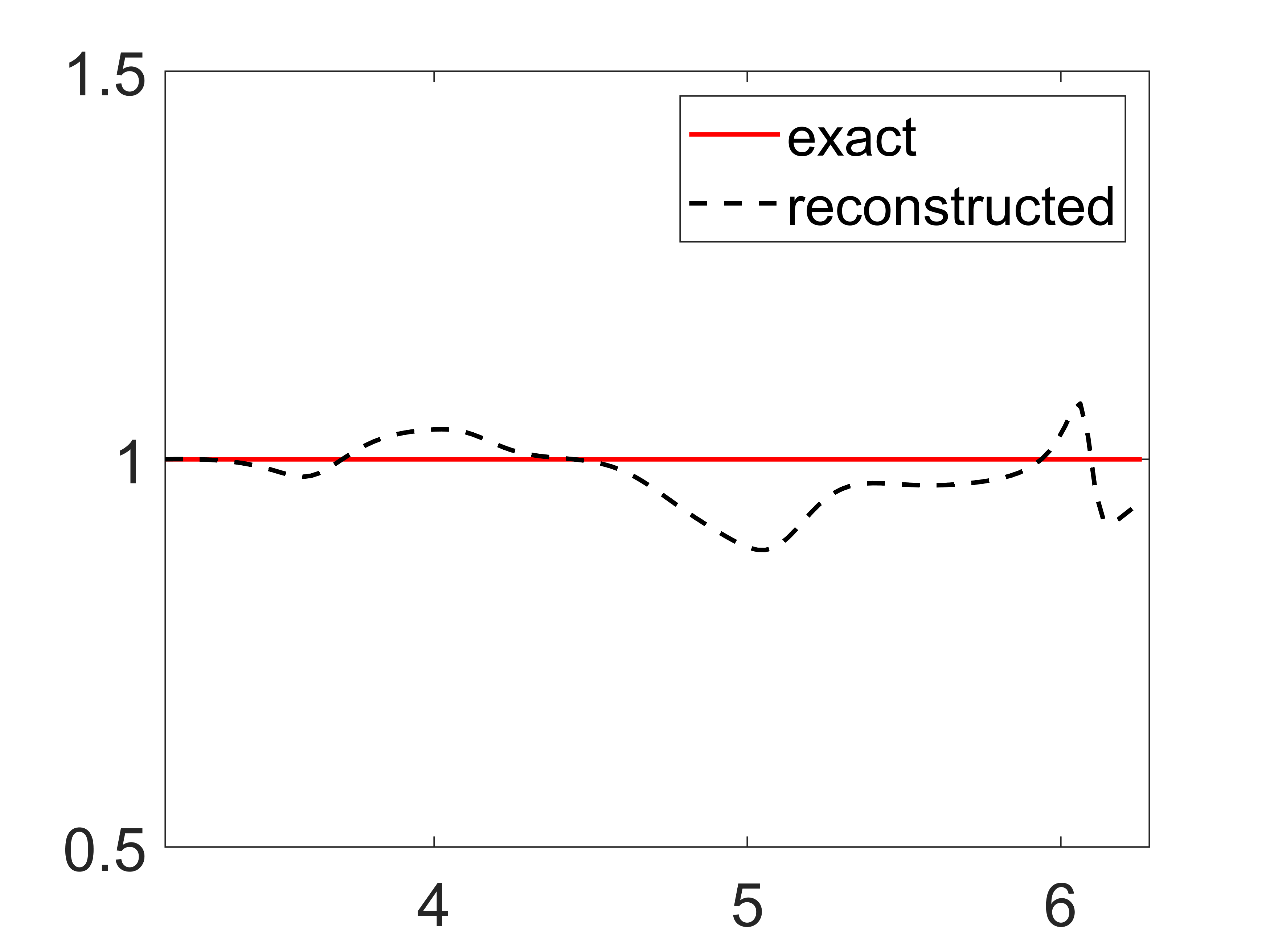}}%
		\subfigure[Noise 5$\%$]{\includegraphics[angle=0,
			width=0.33\linewidth]{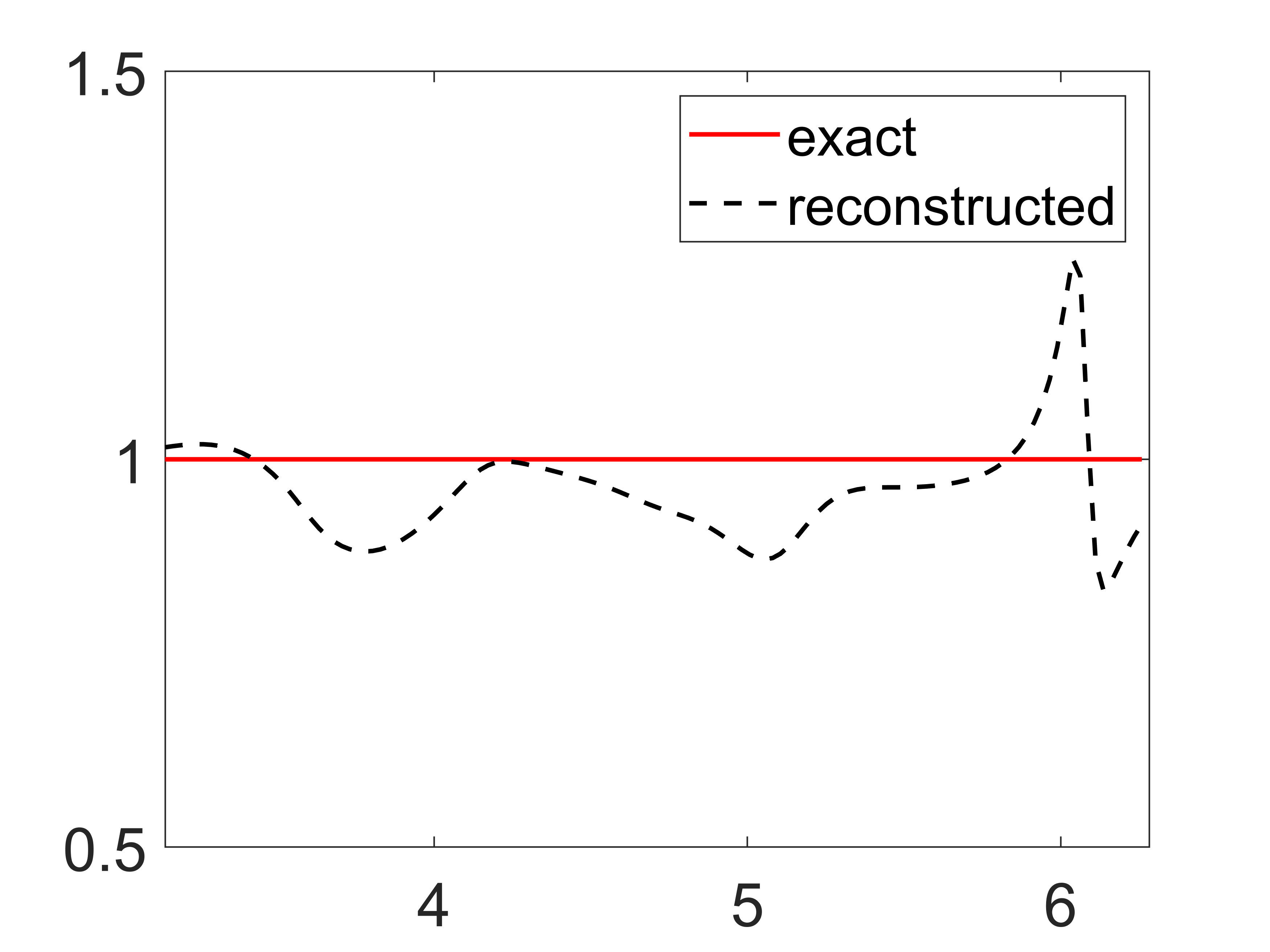}}%
		\caption{Reconstruction of the starfish-shaped domain and the impedance function under different noise levels $\delta\in\{0,1\%,5\%\}$ in Example \ref{EX2}.}\label{figE2:2}
	\end{center}
\end{figure}

\begin{example}\label{EX0}
	In this example, we choose $D$ to be a circle centered at the origin with radius $1.2$.
	The impedance function $\chi$ is given by
	\begin{equation}
		\chi(\vartheta)=\sin^2\vartheta,~~\vartheta\in[0,2\pi),
	\end{equation}
	and the input boundary function $\bm{g}(\bm x)=(g_1,g_2)^\top$ is given by
	\begin{equation}
		g_1(\bm x(\vartheta))=g_2(\bm x(\vartheta))=\sin^2\vartheta,~~\vartheta\in[0,2\pi).
	\end{equation}
\end{example}
Take the virtual boundary to be $\partial B=\{\bm{x}:~|\bm{x}|=7\}$. The initial guesses for the boundary and the impedance function are chosen to be a circle centered at the origin of radius $0.6$ and a constant function $\chi_0=0.5$, respectively.

We fix the frequency to be $\omega=2$. The iterative steps and the regularization parameter determined by Morozov discrepancy principle are displayed in Table \ref{table2}. Figure \ref{figE0:2} exhibits the reconstruction for the missing boundary as well as the impedance function subject to different noise levels $\delta\in\{0,1\%,5\%\}$. We can see that the reconstruction are satisfactory and the number of the iterative steps will not increase obviously compared with the reconstruction utilizing the noise-free data.
\begin{table}[htbp]
	\caption{The iterative steps and the regularization parameter for Example \ref{EX0}.}
	\begin{center}
		\begin{tabular}{ccccccc}
			\hline
			${\rm Noise~level}$                   &${\rm Iterative~steps}$      &${\rm Regularization~ parameter}$  \\
			\hline
			$0$      &$44$               &$9.8186$e$-17$      \\
			$1\%$       &$47$             &$1.0717$e$-05$  \\
			$5\%$      &$50$              & $8.4172$e$-05$   \\
			\hline
		\end{tabular}\label{table2}
	\end{center}
\end{table}

\begin{figure}[htp]
	\begin{center}
		\subfigure[]{\includegraphics[angle=0,
			width=0.33\linewidth]{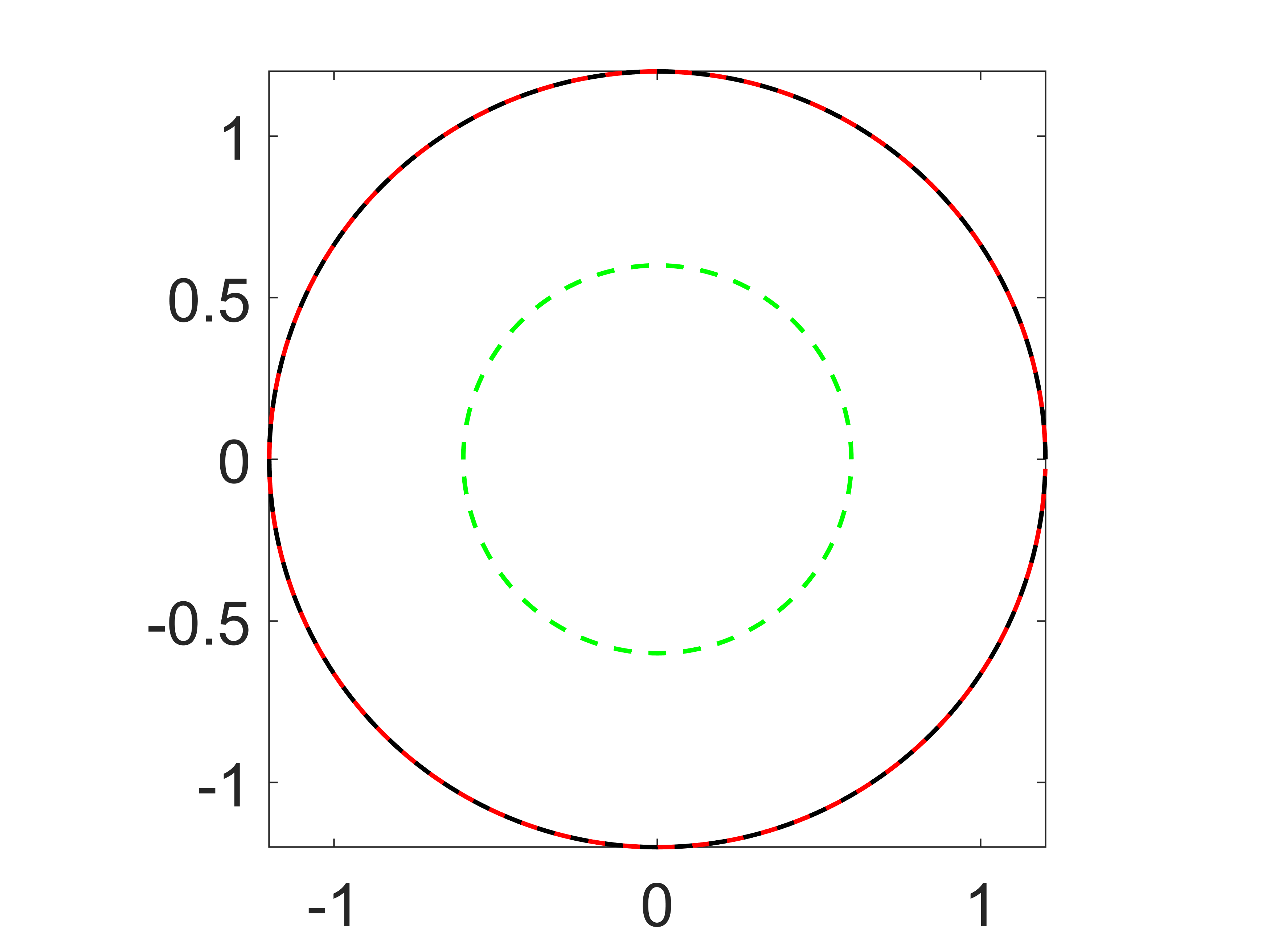}}%
		\subfigure[]{\includegraphics[angle=0,
			width=0.33\linewidth]{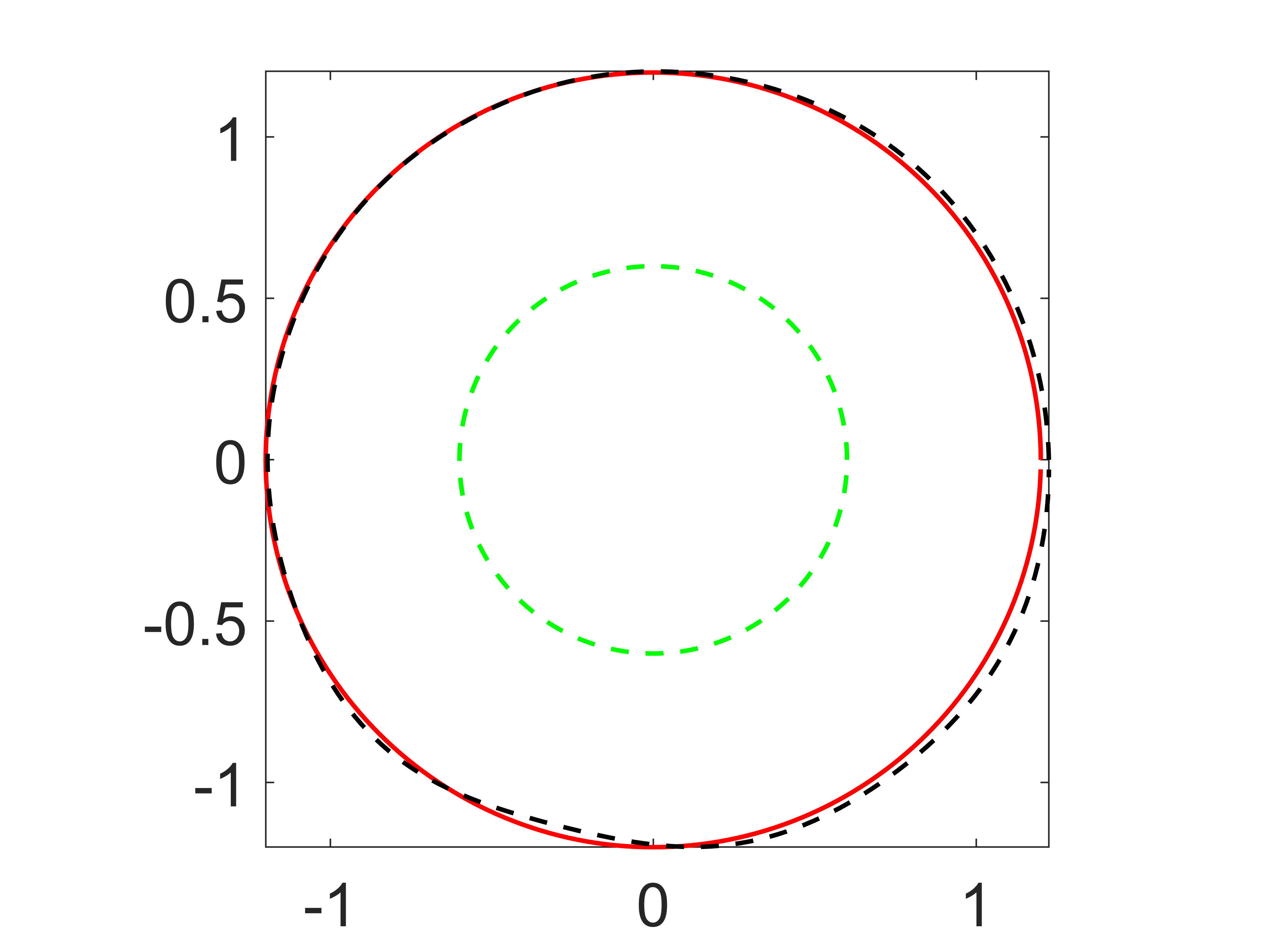}}%
		\subfigure[]{\includegraphics[angle=0,
			width=0.33\linewidth]{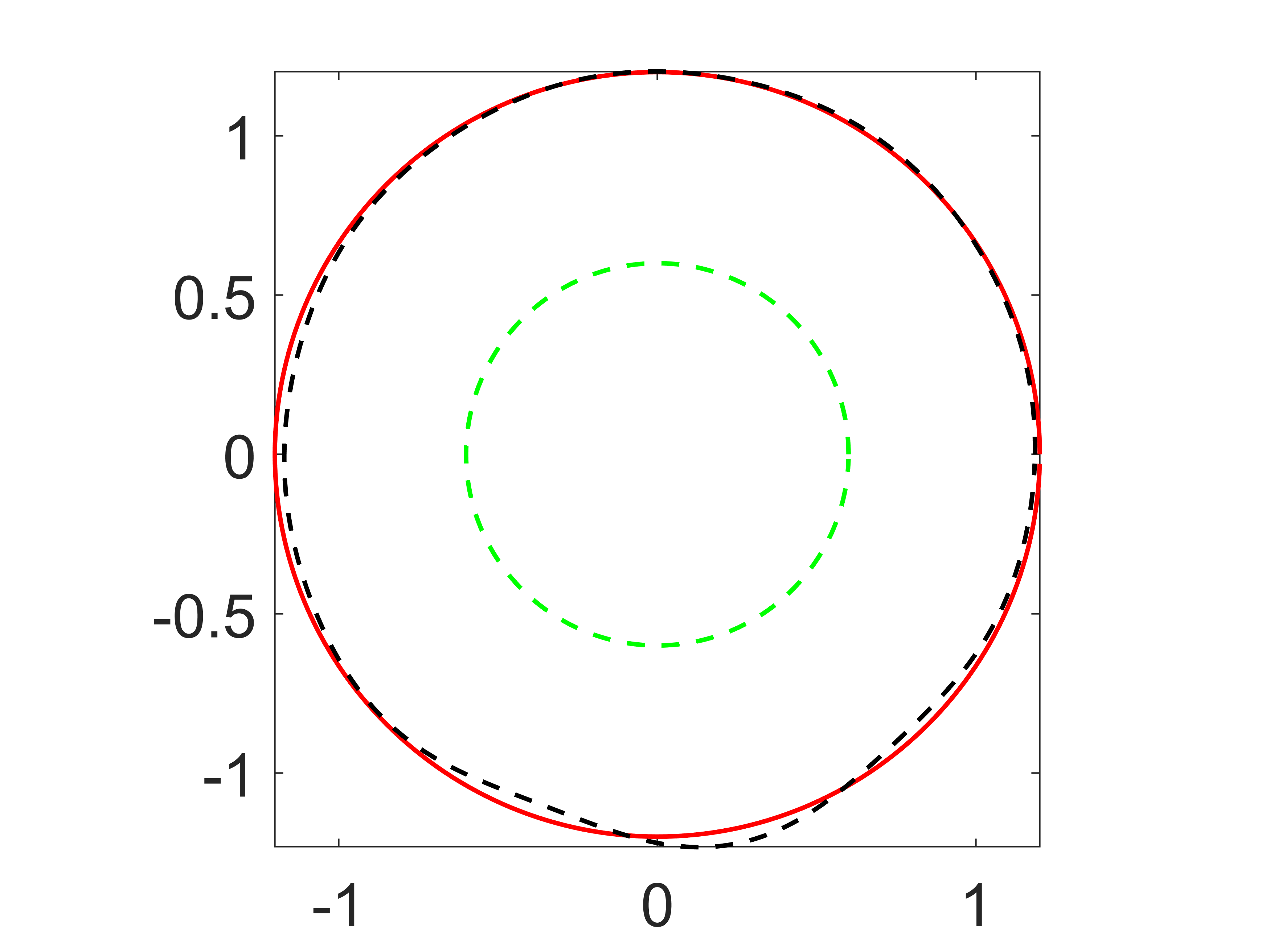}}\\%
		\subfigure[]{\includegraphics[angle=0,
			width=0.33\linewidth]{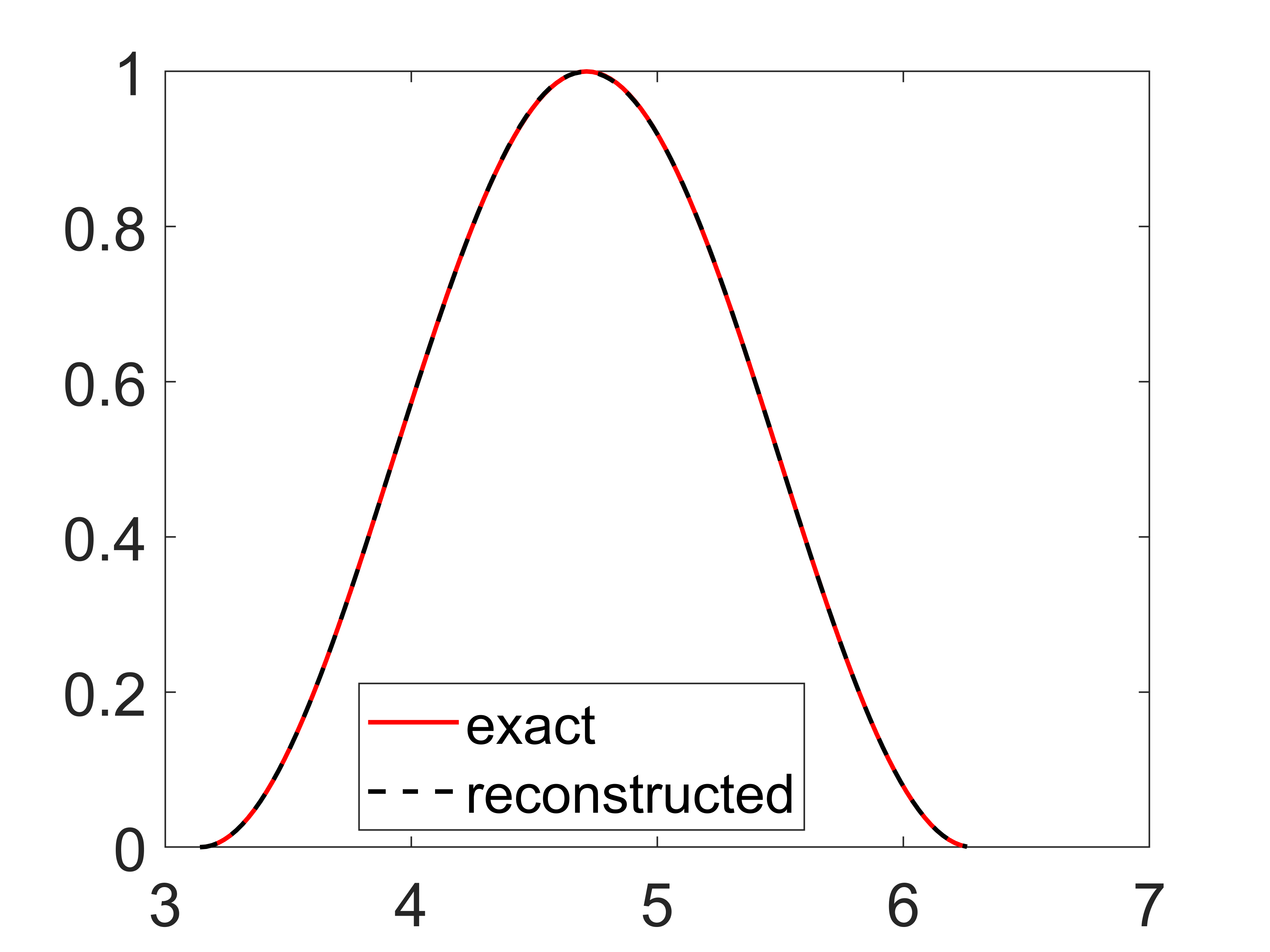}}%
		\subfigure[]{\includegraphics[angle=0,
			width=0.33\linewidth]{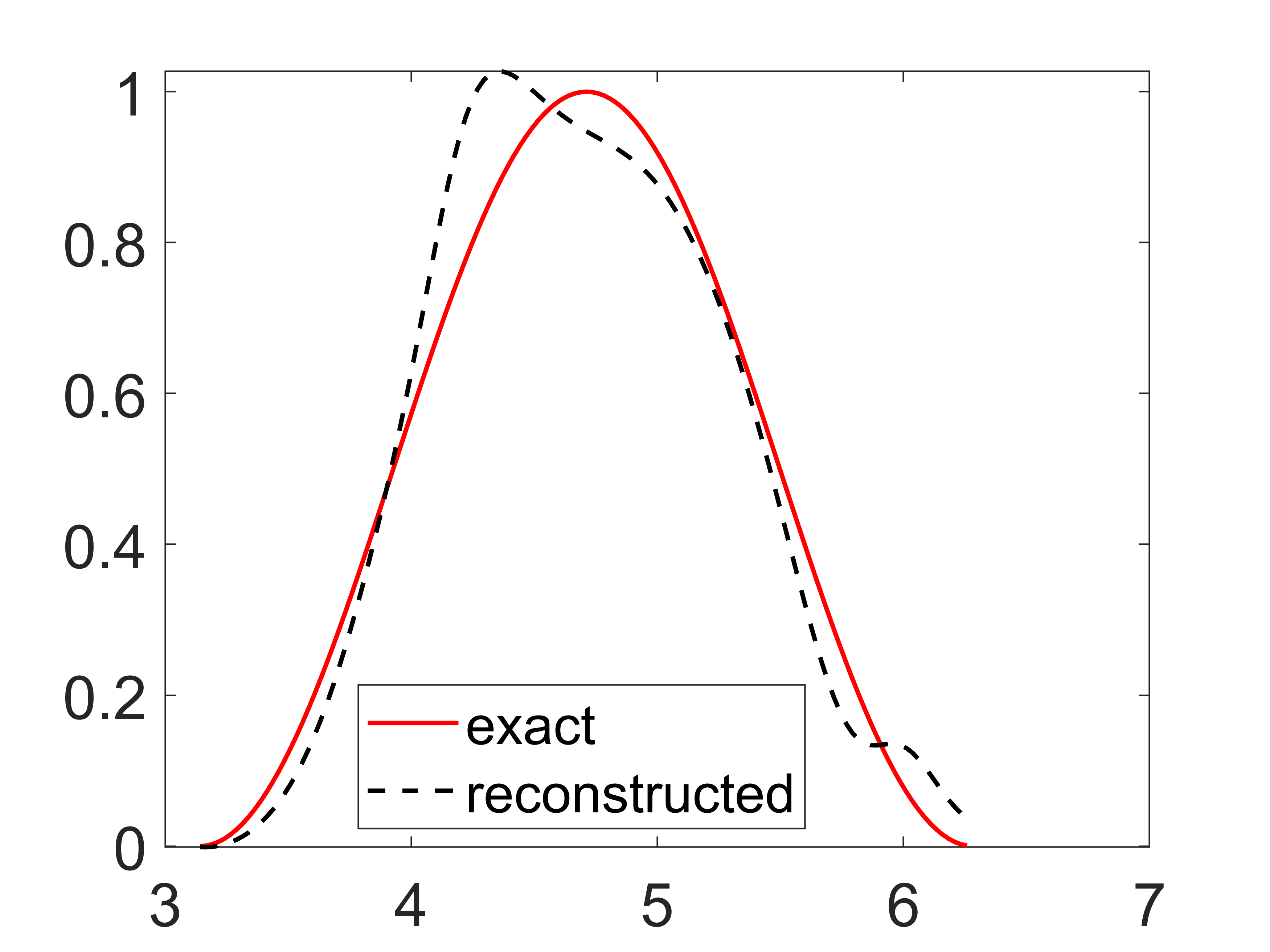}}%
		\subfigure[]{\includegraphics[angle=0,
			width=0.33\linewidth]{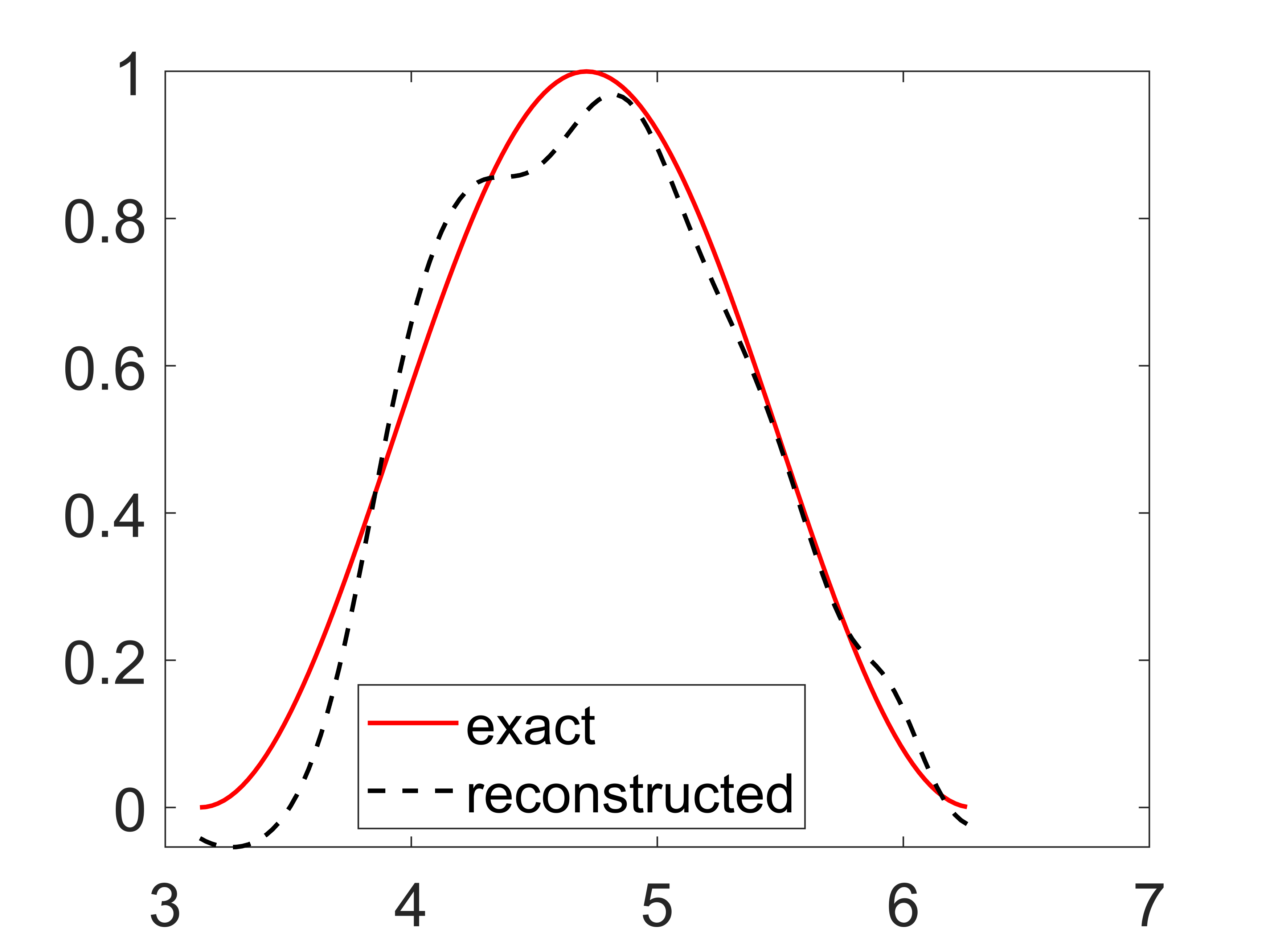}}\\%
		\caption{Reconstruction of the shape and the impedance function under different noise levels in Example \ref{EX0} (From left to right:  $\delta=0,1\%,5\%,$ respectively).}\label{figE0:2}
	\end{center}
\end{figure}


\section{Concluding remarks}
In this paper, we have studied the {inverse elastic problem} by a Newton-type iterative algorithm. We divide the  problem under consideration into two parts. The first part is to solve a Cauchy problem through using an indirect integral equation method combining a regularization technique. The second part is to simultaneously recover the elastic impedance and the shape by a Newton-type iterative method. Several numerical experiments are conducted to verify the effectiveness of the proposed method. In the further work, we shall focus on the extension to three dimensions.

\section*{Acknowledgements}
Yao Sun supported by the Natural Science Foundation of China (No: 11501566) and Tianjin Education Commission Research Project (No: 2022KJ072). Yan Chang and Yukun Guo are supported by NSFC grating 11971133.


\begin{thebibliography}{99}
	
	\bibitem{6}
	G. Alessandrini, A. Morassi and E. Rosset,
	Detecting cavities by electrostatic boundary measurements.
	{\em Inverse Problems}, 18(5):1333, 2002.
	
	
	
	\bibitem{1}
	H. Ammari, E. Bretin, J. Garnier, H. Kang, H. Lee and
	A. Wahab,
	In {\em {Mathematical Methods in Elasticity Imaging}}. Princeton
	University Press, 04 2015.
	
	\bibitem{4}
	H. Ammari, E. Bretin, J. Garnier, H. Kang, H. Lee and
	A. Wahab,
	In {\em {Theory of Elasticity}}. 1986.
	
	
	
	
	\bibitem{Colton2013} D. Colton and R. Kress, \emph{Inverse Acoustic and Electromagnetic Scattering Theory}, 4th edn (New York: Springer), 2019.
	
	\bibitem{17}
	F. Cakoni and R. Kress,
	{Integral equations for inverse problems in corrosion detection from
		partial cauchy data.}
	{\em Inverse Problems and Imaging}, 1:229--245, 2007.
	
	
	
	\bibitem{18}
	F. Cakoni, R. Kress and C. Schuft,
	{Simultaneous reconstruction of shape and impedance in corrosion
		detection.}
	{\em Methods and Applications of Analysis}, 17: 357--378, 2010.
	
	
	\bibitem{era} Y. Chang, Y. Guo, Simultaneous recovery of an obstacle and its excitation sources from near-field scattering data, \emph{Electronic Research Archive}, 30: 1296--1321, 2022.
	
	
	\bibitem{Chang23}Y. Chang, Y. Guo, H. Liu, D. Zhang, Recovering source location, polarization, and shape of obstacle from elastic scattering data,
	\emph{Journal of Computational Physics},
	489:112289, 2023.
	
	\bibitem{RTM} Z. Chen, G. Huang, Reverse time migration for extended obstacles: Elastic waves (in Chinese). \emph{Sci. Sin. Math.}, 45: 1103--1114, 2015.
	
	
%
%
%
%
	
	
	
	\bibitem{8}
	D. Fasino and G. Inglese,
	{Discrete methods in the study of an inverse problem for Laplace's
		equation}.
	{\em IMA Journal of Numerical Analysis}, 19(1):105--118, 1999.
	
	
	
%
%
	
	
	\bibitem{12}
	G. Hu, A. Kirsch and M. Sini,
	Some inverse problems arising from elastic scattering by rigid
	obstacles.
	{\em Inverse Problems}, 29(1):015009, 2013.
	
	\bibitem{13}
	G. Hu, A. Mantile, M. Sini and T. Yin,
	Direct and inverse time-harmonic elastic scattering from point-like
	and extended obstacles.
	{\em Inverse Problems and Imaging}, 14(6):1025--1056, 2020.
	
	\bibitem{Isakov} V. Isakov, On uniqueness of obstacles and boundary conditions from restricted dynamical and scattering data. \emph{Inverse Problems and Imaging}, 2(1):151--165, 2008.
	
%
	\bibitem{lxd} X. Ji, X. Liu,
	Direct sampling methods for inverse elastic scattering problems. {\em Inverse Problems}, 34: 035008, 2018.
	
%
%
	
	
	\bibitem{7}
	R. Kress,
	Uniqueness and numerical methods in inverse obstacle scattering.
	{\em Journal of Physics: Conference Series}, 73(1):012003, 2007.
	
	\bibitem{Kupradze79}V.D. Kupradze, \emph{Three-dimensional problems of the mathematical theory of elasticity and
		thermoelasticity}. Amsterdam: North-Holland; 1979.
	
	
	
	\bibitem{KressIP01}R. Kress, W. Rundell, {Inverse scattering for shape and impedance}. \emph{Inverse Problems}, 17(4): 1075--1085, 2001.
	
	\bibitem{Kress2005IP}R. Kress and W. Rundell, Nonlinear integral equations and the iterative solution for an inverse boundary value problem. \emph{Inverse Problems}, 21(4):1207, 2005.
	
	\bibitem{KressJIE08}R. Kress, W. Rundell, {Inverse scattering for shape and impedance revisited}. \emph{Journal of Integral Equations and Applications}, 30(2):293--311, 2018.
	
	
	
	
	
	
	\bibitem{Liu2019}J. Liu, X. Liu, J. Sun, Extended sampling method for inverse elastic scattering problems using one incident wave. \emph{SIAM Journal on Imaging Sciences}, 12(2):874--892, 2019.
	
	
	F. L. Lou\"{e}r,
	A domain derivative-based method for solving elastodynamic inverse
	obstacle scattering problems.
	{\em Inverse Problems}, 31(11):115006, 2015.
	
	
	\bibitem{McLean} W. McLean, \emph{Strongly Elliptic Systems and Boundary Integral Equations}, Cambridge: Cambridge University
	Press, 2000.
	
	\bibitem{5}A. Morassi and E. Rosset,
	Stable determination of cavities in elastic bodies.
	{\em Inverse Problems}, 20(2):453, 2004.
	
	\bibitem{Sun14}{Y. Sun, F. Ma, D. Zhang, An integral equations method combined minimum norm
		solution for 3D elastostatics Cauchy problem. \emph{Computer Methods in Applied Mechanics and Engineering}, 271: 231--252, 2014.}
	
	
	\bibitem{Sun172}Y. Sun, L. Marin, An invariant method of fundamental solutions for two-dimensional isotropic linear elasticity. \emph{International Journal of Solids and Structures}, 117: 191--207, 2017.
	
	\bibitem{Sun17}{ Y. Sun, F. Ma, X. Zhou, An Invariant Method of Fundamental Solutions for the Cauchy Problem in Two-Dimensional Isotropic Linear Elasticity.
		\emph{Journal of Scientific Computing}, 64: 197--215, 2015.}
	
	
	\bibitem{Sun24} Y. Sun, Y. Wang, A highly accurate indirect boundary integral equation solution for three
	dimensional elastic scattering problem. \emph{Eng. Anal. Bound. Elem.} 159: 402--417, 2024.
%
%
%
%
%
\end{thebibliography}
\end{document}